\documentclass[a4paper,12pt,oneside]{article}

\usepackage[hidelinks]{hyperref}
\usepackage{amsmath}
\usepackage{amssymb}
\usepackage{amsthm}
\usepackage[english]{babel}
\usepackage{graphicx}
\usepackage[lmargin=25mm,rmargin=25mm,top=25mm,bottom=30mm,paperwidth=210mm,paperheight=11in]{geometry}
\usepackage{mathrsfs}
\usepackage{pifont}
\usepackage{tikz}
\usepackage[utf8]{inputenc}
\usepackage[T1]{fontenc}
\usepackage{lmodern}

\newtheorem{theorem}{Theorem}[section]
\newtheorem*{theorem*}{Theorem}
\newtheorem{lemma}[theorem]{Lemma}
\newtheorem{proposition}[theorem]{Proposition}
\newtheorem{corollary}[theorem]{Corollary}
\newtheorem{definition}[theorem]{Definition}
\newtheorem*{definition*}{Definition}
\newtheorem*{remark}{Remark}

\usepackage[titletoc,title]{appendix}

\newcommand{\E}{\mathbb{E}}
\newcommand{\N}{\mathbb{N}}

\renewcommand{\P}{\mathbb{P}}
\newcommand{\R}{\mathbb{R}}
\newcommand{\Z}{\mathbb{Z}}

\newcommand{\1}{1\hspace{-1.3mm}1}
\newcommand{\Var}{\operatorname{Var}}
\newcommand{\Cov}{\operatorname{Cov}}

\usepackage{pgfplots}
\usetikzlibrary{calc}
\usetikzlibrary{patterns}
\definecolor{myblue}{rgb}{0,0.5,1}
\definecolor{mypink}{rgb}{1,0,1}
\definecolor{mygreen}{rgb}{0,0.6,0}
 \pgfdeclarepatternformonly{north west spaced lines}{\pgfqpoint{-1pt}{-1pt}}{\pgfqpoint{10pt}{10pt}}{\pgfqpoint{9pt}{9pt}}%
{
  \pgfsetlinewidth{0.4pt}
  \pgfpathmoveto{\pgfqpoint{0pt}{10pt}}
  \pgfpathlineto{\pgfqpoint{10.1pt}{-0.1pt}}
  \pgfusepath{stroke}
}

 \makeatletter
\def\parsept#1#2#3{%
    \def\nospace##1{\zap@space##1 \@empty}%
    \def\rawparsept(##1,##2){%
        \edef#1{\nospace{##1}}%
        \edef#2{\nospace{##2}}%
    }%
    \expandafter\rawparsept#3%
}
\makeatother

\begin{document}

\title{Front evolution of the Fredrickson-Andersen one spin facilitated model%\footnote{This work has been supported by the ERC Starting Grant 680275 MALIG and by the ANR projects LSD (ANR-15-CE40-0020), MALIN (ANR-16-CE93-0003) and PPPP (ANR-16-CE40-0016)}
}

\author{Oriane Blondel\footnote{Univ Lyon, CNRS, Université Claude Bernard Lyon 1, UMR5208, Institut  Camille Jordan, F-69622 Villeurbanne, France. {\it blondel@math.univ-lyon1.fr}}, \, Aurelia Deshayes\footnote{Laboratoire de Probabilités, Statistiques et Modélisation, UMR8001, Université Paris Diderot,  Sorbonne Paris Cité, CNRS, F-75013 Paris, France. {\it deshayes@lpsm.paris}} \, and Cristina Toninelli\footnote{Laboratoire de Probabilités, Statistiques et Modélisation, UMR8001, Université Paris Diderot,  Sorbonne Paris Cité, CNRS, F-75013 Paris, France. {\it cristina.toninelli@upmc.fr}}}

%\thanks{This work has been supported by the ERC Starting Grant 680275 ``MALIG'', ANR ``LSD'', ANR ``MALIN'' and ANR-16-CE40-0016 ``PPPP''}

%\subjclass[2010]{Primary {60K35}, secondary 60J27}

%\keywords{Kinetically constrained models, invariant measure, coupling, contact process}

\maketitle

\begin{abstract}
The Fredrickson-Andersen one spin facilitated model (FA-1f) on $\mathbb Z$ belongs to the class of kinetically constrained spin models (KCM).
Each site refreshes with rate one its occupation variable to empty (respectively occupied) with probability $q$ (respectively $p=1-q$),
provided at least one nearest neighbor is empty. 
Here, we study the non equilibrium dynamics of FA-1f started from a configuration entirely occupied on  the left half-line and focus on the evolution of the front, namely the position of the leftmost zero.  We prove, for $q$ larger than a threshold $\bar q<1$, a law of large numbers and a central limit theorem for the front, as well as the convergence to an invariant measure of the law of the process seen from the front.

\end{abstract}



\section{Introduction}
Fredrickson-Andersen one spin facilitated model (FA-1f) belongs to the class of kinetically constrained spin models (KCM).
KCM are interacting particle systems on $\mathbb Z^d$ which have been introduced in physics literature in the '80s (see \cite{RS} for a review) to model the liquid glass transition, 
a major open problem in condensed matter physics.
A configuration of a KCM is given by assigning to each vertex $x\in\mathbb Z^d$ an occupation variable 
$\sigma(x)\in\{0,1\}$, which corresponds to an empty or occupied site, respectively. 
The evolution is given by a Markovian stochastic dynamics of Glauber type. With rate one, each vertex updates its occupation variable to occupied or to empty with probability $p\in[0,1]$ and $q=1-p$, respectively, if the configuration satisfies a certain local constraint.
For the FA-1f model the constraint requires at least one empty nearest neighbor.
Since the constraint to update the configuration at $x$ does not depend on $\sigma(x)$, the dynamics satisfies detailed balance w.r.t.~Bernoulli product measure at density $p$. A key issue is to analyze the large time evolution when we start from a distribution different
from the equilibrium Bernoulli measure. Note that, due to the presence of the constraints, FA-1f dynamics is not attractive, so powerful tools like monotone coupling and censoring inequalities cannot be applied.
Furthermore, convergence to equilibrium is not uniform on the initial condition since completely occupied configurations are blocked under the dynamics. Therefore the study of the large time dynamics is particularly challenging.
In~\cite{BCMRT13} a convergence to equilibrium was proven when the starting distribution is such that the mean distance between nearest empty sites is uniformly bounded and the equilibrium vacancy density $q$ is larger than a threshold $1/2$. 
In~\cite{MV} the result was extended to initial configurations 
with at least one zero, provided $q$ is sufficiently near to one.

Here, we consider FA-1f on $\mathbb Z$ starting from a configuration which has a zero at the origin and is completely occupied in the left half line and we study the evolution of the front, namely the position of the leftmost zero. 
A law of large numbers and a central limit theorem (Theorem~\ref{thm:limit}) for the front, as well as the convergence to an invariant measure of the law of the process seen from the front (Theorem~\ref{thm:coupling}) are proven. Our results hold for $q$ sufficiently large, more precisely for $q>\bar q$ where $\bar q$ is related  to the critical parameter in the threshold contact process. We stress that, even though we are in one dimension, proving the ballistic motion of the front is non trivial. Indeed, due to the non attractiveness  of the dynamics, the classic tool of sub-additivity~\cite{Dur80} cannot be used. Obviously, all the following results will also be true for the position of the rightmost zero in the FA-1f starting from a configuration which has a zero at the origin and is completely occupied in the right half-line.  

The motion of the front has been analyzed in \cite{front_east,GLM15} for another one dimensional KCM, the East model, for which the constraint requires the site at the right of $x$ to be empty: ergodicity of the measure seen from the front, law of large numbers, central limit theorem and cutoff results have been established. A key tool for the East model introduced in \cite{AD} and used in \cite{front_east,GLM15} for the study of the front, is the construction of a distinguished zero, a sort of moving boundary which induces local relaxation to equilibrium. This construction relies heavily on the oriented nature of the East constraint and cannot be extended to FA-1f and to generic KCM. Another consequence of the orientation is that the cutoff  result in~\cite{GLM15} follows immediately from the central limit theorem for the front, while for FA-1f it would involve a more complex argument.

A sketch of the main step of our proof follows.
Our first key result is to prove relaxation to equilibrium far behind the front (Theorem~\ref{thm:decorrelation}).
In order to establish this result, we couple FA-1f dynamics with a sequence of threshold contact processes (Lemma~\ref{couplage1}) where zeros flip to ones at rate $p$ without any constraint, and ones flip to zeros at rate $q$ if and only if there is at least one nearest neighbor zero. The first contact process starts with a zero at the position of the front (namely at the origin), then if the contact process dies we restart a new one from the last killed zero. The threshold contact process is attractive, and it is well known that for $q$ above a threshold $\bar q<1$  the front of the process conditioned on survival moves ballistically. Due to the fact that there is no constraint for the contact process for the move $0\to 1$, we can couple FA-1f and contact trajectories in such a way that FA-1f configurations contain more zeros than contact process configurations. Then, we can use the well-known behavior of the contact process (ballistic motion of the front, shape theorem for the coupled zone) to guarantee enough zeros behind the front of FA-1f. This will be the work in Section~\ref{section:couplingfapc}. The construction is illustrated by the first simulation in Figure~\ref{fig:simu}; we can see on the second one that the ballistic behavior of the FA-1f front seems still valid for some $q\leq \bar q$ but in this regime, the contact process does not give us any useful information. Afterward, we apply the techniques of~\cite{BCMRT13} (Corollary~\ref{cor:relaxation}) to prove relaxation to equilibrium using these zeros in Section~\ref{section:decorrelation}.
Once this result is established, we construct a coupling inspired by \cite{front_east,GLM15} to prove ergodicity of the process behind the front, namely convergence to the unique invariant measure for the process seen from the front (Section~\ref{section:coupling}).
This convergence result allows us to analyze the increments of the front (Section~\ref{section:hypotheses}) and to deduce a law of large numbers. Finally, to study the fluctuations of the front, we generalize the strategy of~\cite{GLM15} which is in turn based on the result of Bolthausen~\cite{bolt} which allows to establish a central limit theorem for random variables which are not stationary but satisfy a proper mixing condition (Theorem~\ref{thm:TCLgal}).

\begin{figure}[hbt]
{\centering
\includegraphics[width=0.45\textwidth,height=0.3\textwidth]{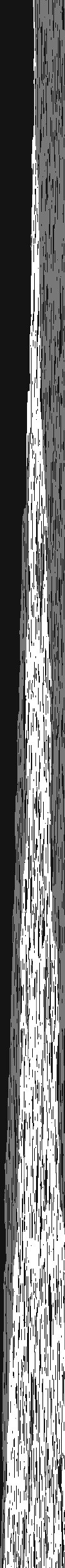}
\includegraphics[width=0.45\textwidth,height=0.3\textwidth]{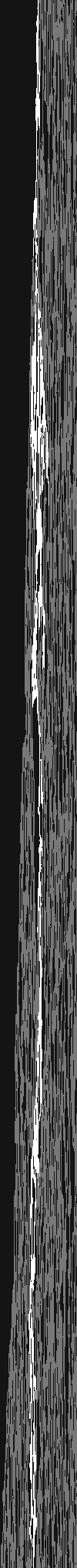}
\par}
\caption{Simulation of FA-1f dynamics (gray points) coupled with restart threshold contact processes (white points). The first one is for $q>\bar q$ and the second one for $q<\bar q$.}
\label{fig:simu}
\end{figure}

\section{Models and main results}
\subsection{The FA-1f process}
Let $\Omega=\{0,1\}^{\Z}$ be the space of configurations and \[
LO_0=\{\sigma\in\Omega:\forall x<0, \sigma(x)=1, \sigma(0)=0\}
\]
be the subspace of configurations with a leftmost zero at the origin. 
For a configuration $\sigma$ for which there exists $x\in\Z$ such that for every $y<x$, $\sigma(y)=1$ and $\sigma(x)=0$, we call $x$ the \emph{front} of configuration $\sigma$ and we denote it by $X(\sigma)$. %\sigma$ such that the front $X(\sigma)=\inf \{x\in\Z: \sigma(x)=0\}$ is well defined. 
For $\Lambda\subset \Z$ and $\sigma\in\Omega$, we denote by $\sigma_{\Lambda}$ the restriction of $\sigma$ to the set $\Lambda$. 
For $\sigma\in\Omega$ and $x\in\Z$, let $\sigma^x$ be the configuration $\sigma$ flipped at site $x$. We denote by $\theta_y$ the space shift operator of vector $y$: $\theta_y\sigma(x)=\sigma(x+y)$. We denote by $\delta^y$ the configuration such that $\delta^y(x)=1$ for all $x\neq y$ and $\delta^y(y)=0$. 

The FA-1f dynamics on $\Omega$ are described by a Markov process with the following generator: for any local function $f$ and any $\sigma\in\Omega$,
\begin{equation}\label{genFA}
\mathcal L f (\sigma)=\sum_{x\in \Z} r(x,\sigma)\left(f(\sigma^x)-f(\sigma)\right)
\end{equation}
where the rate $r(x,\sigma)=c(x,\sigma)(q\sigma(x)+p(1-\sigma(x)))$ is the product between a constraint $c(x,\sigma)$ and a flip rate $q\sigma(x)+p(1-\sigma(x))$. The constraint $c(x,\sigma)=1-\sigma(x-1)\sigma(x+1)$ requires at least one empty neighbor to allow the flip. The parameter $p$ is the rate to update to 1 and the parameter $q=1-p$ is the rate to update to 0. Let $\sigma_t$ be the configuration at time $t$ starting from $\sigma$.  
When there is no confusion, $X(t)$ denotes the front of the configuration $\sigma_t$.

In the following, all constants may depend on $q$.

It is easy to verify that the FA-1f process is reversible w.r.t.\ $\mu:=\mathrm{Ber}(p)^{\otimes \Z}$.

\subsection{An auxiliary process: the threshold contact process}\label{section_pc}
We introduce a threshold contact process where the $0$'s are the infected points and the $1$'s are the healthy ones. Its dynamics on $\Omega$ is given by a Markov process with the following generator: for any local function $f$ and any $\eta\in\Omega$,
 \[
 \mathcal L' f (\eta)=\sum_{x\in \Z} r'(x,\eta)(f(\eta^x)-f(\eta))\]
 where $r'(x,\eta)=c(x,\eta)q\eta(x)+ p(1-\eta(x))$. 
In this model, the constraint only applies to a flip $1\rightarrow 0$. Namely, a $1$ particle can flip to $0$ only if one of its neighbor is $0$; we interpret the state 0 as an infection that propagates by contact. The flip from 0 to 1 is a spontaneous recovery. Let $\eta_t$ be the configuration at time $t$ starting from $\eta$.

We define the extinction time of the threshold contact process by
\begin{align}
\label{extinctiontimePC}
\tau (\eta_{\cdot})=\inf\{t>0, \forall x\in\Z,~\eta_t(x)=1 \},
\end{align}
that is the first time when the threshold contact process has no more zero (\textit{i.e.} infected site). This state is absorbing for the process.
If $\eta_0$ has a finite number of infected sites, $\tau (\eta_{\cdot})$ can be finite. For the FA-1f process, the corresponding extinction time is always infinite because a single zero can not disappear. If the event $\{\tau(\eta_{\cdot})=+\infty\}$ occurs we say that the contact process survives.

To ensure that the threshold contact process survives with positive probability, we need to suppose that $\frac{q}{p}>\lambda^{\scriptscriptstyle{\text{TCP}}}_c(\Z)$ where the critical parameter of the threshold contact process $\lambda^{\scriptscriptstyle{\text{TCP}}}_c(\Z)$ has an approximate value of $1.74$ (\textit{cf.}~\cite{BD88}). In the following we will suppose a stronger condition which is $\frac{q}{p}>2\lambda_c(\Z)$ where $\lambda_c(\Z)$ is the critical parameter of the classical contact process and has an approximate value of $1.65$ (\textit{cf.}~\cite{brower}). This hypothesis allows us to use all the classical contact process estimates (\textit{cf.} Appendix~\ref{annexe_cp}) instead of having to reestablish them %which are not clearly proved 
for the threshold contact process in the intermediate regime $\lambda\in\left(\lambda^{\scriptscriptstyle{\text{TCP}}}_c(\Z),2\lambda_c(\Z)\right]$. So, in this paper, we will consider $q>\overline{q}$ where
\[
\overline{q}=\frac{2\lambda_c(\Z)}{1+2\lambda_c(\Z)}.
\]
This corresponds approximately to taking $q\gtrsim 0.76$ rather than $q\gtrsim 0.63$.

\subsection{Main results}

Now we have the tools to enounce precisely the theorems that we will prove in this paper. The first one is the ergodicity of the process seen from the front. It will be proved thanks to a major coupling in Section~\ref{section:coupling}.
\begin{theorem}
\label{thm:coupling}
Let $q>\overline{q}$. The process seen from the front has a unique invariant measure $\nu$ and, starting from every $\sigma\in LO_0$, it converges in distribution to $\nu$ in the following sense: there exist $d^*>0$ and $c>0$ (independent of $\sigma$) such that for $t$ large enough
\[
\|\tilde{\mu}_t^{\sigma}-\nu\|_{[0,d^*t]}\leq e^{-ce^{(\log t)^{1/4}}},
\]
where $\tilde{\mu}_t^\sigma$ is the distribution of the configuration seen from the front at time $t$ starting from $\sigma$, i.e.\@ $\theta_{X(\sigma_t)}\sigma_t$, and $\|\pi-\pi'\|_\Lambda$ denotes the total variation distance between the marginals of $\pi$ and $\pi'$ on $\Lambda$.
\end{theorem}

\begin{remark} For every $\alpha >0$, the velocity of convergence $e^{-e^{c(\log t)^{1/4}}}$ is less good than $e^{-t^{\alpha}}$ (which was the velocity obtained in the East case by~\cite{GLM15} for some $\alpha<1/2$) but it is better than $e^{-(\log t)^{1/\alpha}}$ (and in particular better than any polynomial velocity).
\end{remark}

The second one is a law of large numbers and a central limit theorem for the front. 
\begin{theorem}\label{thm:limit}
Let $q>\overline{q}$. There exists $s=s(q)$ and $v=v(q)$ such that for all $\sigma \in LO_0$
\begin{eqnarray}
\label{eq:lgn}
\frac{X(\sigma_t)}{t} & \xrightarrow[t\to\infty]{} & v ~~~~~\P_{\sigma}-\text{almost surely}, \\
\label{eq:tlc}
\frac{X(\sigma_t)-vt}{\sqrt{t}} & \xrightarrow[t\to\infty]{d} & \mathcal{N}(0,s^2) ~~~~~ w.r.t.~\P_{\sigma},
\end{eqnarray}
where $v=p\cdot\nu[\tilde{\sigma}(1)=0]-q$ is negative.
\end{theorem}

\subsection{Graphical construction and basic coupling}\label{graphical}

We briefly recall here the graphical construction for the FA-1f and contact process, which allows to construct the two processes on the same probability space and to compare them pointwise. Troughout the text, $\N$ denotes the set of positive integers. Let $\mathcal{C}=\left(B_{x,n}, E_{x,n}\right)_{x\in\Z,n\in\N}$ be a collection of independent random variables, where for all $(x,n)\in\Z\times\N$, $B_{x,n}\sim\mathrm{Ber}(p)$, $E_{x,n}\sim\mathrm{Exp}(1)$. These variables are interpreted as follows: with each site $x\in\Z$ we associate a sequence of exponential clock rings given by the $\sum_{k=1}^nE_{x,k},\ n\in\N$, and with the clock ring at time $\sum_{k=1}^nE_{x,k}$ we associate the Bernoulli variable $B_{x,n}$. Starting from configurations $\sigma,\eta\in\Omega$, we construct a FA-1f process $(\sigma_t)_{t\geq 0}$ and a (threshold) contact process $(\eta_t)_{t\geq 0}$ using the same collection $\mathcal{C}$. When a clock rings at $x$, we update each process at this site to the associated Bernoulli variable, provided the constraint of the process is satisfied (for instance, if the Bernoulli variable is $1$, the contact process automatically updates, while the FA-1f needs at least one empty neighbor). We call $(\sigma_t,\eta_t)_{t\geq 0}$ \emph{the basic coupling started from $(\sigma,\eta)$ using $\mathcal{C}$.} We denote by $\P,\E$ the associated probability and expectation. This probability space allows us to construct the processes with any initial configuration simultaneously. We will denote by $\P_{\sigma,\eta}$ (respectively $\P_{\sigma}$, $\P_{\eta}$) the associated probability to the initial configuration $(\sigma,\eta)$ (respectively the projected probabilities).

The following property will be our main tool to guarantee a minimal quantity of zeros in the FA-1f process.
\begin{lemma}\label{order}
If $\sigma\leq\eta$ (pointwise), with the above construction, we have a.s. 
\begin{equation*}
\forall t\geq 0,\quad\sigma_t\leq\eta_t.
\end{equation*}
\end{lemma}

%namely, the FA-1f process has always more 0's particles than the contact process.

\begin{proof}
It is not difficult to check that every possible transition preserves the order. 
\end{proof}

It will also be useful to define a (space) shift operation on the collection $\mathcal{C}$ by 
\[
\theta_y\mathcal{C}=\left(B_{y+x,n}, E_{y+x,n}\right)_{x\in\Z,n\in\N}.
\]

To quantify the amount of zeros we will introduce the following event. A $0$-gap is an interval without zeros. 

\begin{definition}\label{eventH} Let $L,M,l$ be three positive numbers. 
\begin{align*}
\mathcal{H}(L,L+M,l)&=\{\text{There is no 0-gap of length $l$ in the box }[L,L+M]\}\\
&=\{\sigma:\forall y\in[L ,L+M-l+1],~ \exists z\in [y, y+l-1],~ \sigma(z)=0\}.
\end{align*}
\end{definition}
If $l\leq M+1$, it is also equivalent to say that there is at least one zero in the box and the maximal distance between two zeros in the box $[L,L+M]$ is less than $l$.

\subsection{Finite speed propagation}%AT MOST LINEAR GROWTH
Classically, this type of graphical construction implies finite speed of propagation in the following sense. For $x,y\in\Z$ and $t>0$, we denote by 
\begin{align*}
F(x,y,t)&=\{\text{before time }t\text{ there is a sequence of successive rings linking }x\text{ to }y\},\\
\tilde{F}(x,y,t)&=\{\exists~ z\text{ between }x\text{ and }y\text{ s.t.\ }F(x,z,t)\cap F(y,z,t)\}.
\end{align*}
Above, a sequence of successive rings linking $x$ and $y$ means that (if e.g.\ $x<y$) there is a clock ring at site $x$, and then at site $x+1$ and so on until $y$ is reached. Standard results on Poisson point processes imply the following.
\begin{lemma}\label{finitespeedpropagation}
There exists a constant $\overline{v}$ (which is bigger than 1) such that, for every $t$ and $x,y\in\Z$ such that $|x-y|\geq \overline{v} t$ then
\begin{align*}
\P(F(x,y,t))\leq \P({\tilde{F}}(x,y,t))\leq e^{-|x-y|}.
\end{align*} 
\end{lemma}

This has immediate implications for the maximal velocity of the front in the FA-1f process and for the propagation of the contact process.

\begin{corollary}
For all $\sigma\in LO_0$, $c\geq \overline{v}$,
\begin{align}
\P_\sigma(|X(\sigma_t)|\geq ct)&\leq e^{-ct},\label{atmostlinearfa}\\
\text{and}\quad\P_{\delta^0}(\exists y\  \colon|y|\geq ct\text{ and } \eta_t(y)=0) &\leq e^{-c t}.\label{atmostlinearpc}
\end{align}
\end{corollary}

\section{Relaxation in FA-1f}

In this section, we collect results about the relaxation to equilibrium (\textit{i.e.}\ to $\mu=\mathrm{Ber}(p)^{\otimes\Z}$) of the FA-1f process. These can be deduced from the proofs in \cite{BCMRT13} as explained below.

\begin{proposition}\label{relaxation_lemma}
Let $q>1/2$ and $t>0$. Let $K>0$, $f$ a bounded function with support contained in $[-K,K]$ such that $\mu(f)=0$. 
Define $\Lambda:=[-K-\overline{v}t,K+\overline{v}t]\cap\Z$. Decompose $\Lambda=\sqcup_{i=1}^n\Lambda_i$ with $\Lambda_i$ disjoint intervals in $\Z$ such that for all $i=1,\ldots,n$ $|\Lambda_1|\leq|\Lambda_i|\leq 2|\Lambda_1|$. Choose $\theta>1$ satisfying $\frac{\theta}{\theta+1}<q$. Then there exists $c=c(q,\theta)>0$ such that for all initial configuration $\sigma\in\mathcal{H}(-K-\overline{v}t,K+\overline{v}t,|\Lambda_1|/8)$ , % le 8 peut etre remplacé par un truc > 4
\begin{multline*}
|\E_\sigma[f(\sigma_t)]|\leq c\|f\|_\infty\left(\exp\left(-c^{-1}t+c|\Lambda|e^{-\frac{t}{c|\Lambda_1|}}\right)+tn|\Lambda|\theta^{-|\Lambda_1|/4}\left(c+\theta^{|\Lambda_1|/8}\right)\right.\\
\left.+\ ne^{-q|\Lambda_1|}+|\Lambda|e^{-t/3}\vphantom{e^{-\frac{t}{c|\Lambda_1|}}}
\right).
\end{multline*}
%
%There exist $c<\infty$, $\delta>0$ such that for any $\sigma\in\lbrace 0,1\rbrace^{\Z}$ in which the maximal distance between zeros is not greater than $\delta t$ in $\Lambda:=[-K-100t,K+100t]\cap\Z$, we have for all $t\geq 0$
%\begin{equation}
%|\E_\sigma[f(\sigma_t)]|\leq c\|f\|_\infty e^{-t/c}.
%\end{equation}
\end{proposition}
A similar result holds for the FA-1f dynamics in finite volume.
\begin{proposition}\label{relaxation_finitevolume}
Let $q>1/2$ and $t>0$. Let $K>0$, $f$ a bounded function with support contained in $\Lambda:=[-K,K]$ such that $\mu(f)=0$. 
Decompose $\Lambda=\sqcup_{i=1}^n\Lambda_i$ with $\Lambda_i$ disjoint intervals in $\Z$ such that for all $i=1,\ldots,n$ $|\Lambda_1|\leq|\Lambda_i|\leq 2|\Lambda_1|$. Choose $\theta>1$ satisfying $\frac{\theta}{\theta+1}<q$. Then there exists $c=c(q,\theta)>0$ such that for all initial configuration $\sigma\in\mathcal{H}(-K,K,|\Lambda_1|/8)$,
\begin{multline*}
|\E_\sigma[f(\sigma^\Lambda_t)]|\leq c\|f\|_\infty\left(\exp\left(-c^{-1}t+c|\Lambda|e^{-\frac{t}{c|\Lambda_1|}}\right)+tn|\Lambda|\theta^{-|\Lambda_1|/4}\left(c+\theta^{|\Lambda_1|/8}\right)\right.\\
\left.+\ ne^{-q|\Lambda_1|}+|\Lambda|e^{-t/3}\vphantom{e^{-\frac{t}{c|\Lambda_1|}}}
\right),
\end{multline*}
where $(\sigma^\Lambda_t)_{t\geq 0}$ denotes the FA-1f process in $\Lambda$ with zero boundary condition.

%There exist $c<\infty$, $\delta>0$ such that for any $\sigma\in\lbrace 0,1\rbrace^{\Z}$ in which the maximal distance between zeros is not greater than $\delta t$ in $\Lambda:=[-K-100t,K+100t]\cap\Z$, we have for all $t\geq 0$
%\begin{equation}
%|\E_\sigma[f(\sigma_t)]|\leq c\|f\|_\infty e^{-t/c}.
%\end{equation}
\end{proposition}
\begin{proof}[Proof of Proposition~\ref{relaxation_lemma}]
The proof of this result is essentially contained in \cite{BCMRT13}. We reproduce here the arguments that we need.

First, let $(\sigma^\Lambda_t)_{t\geq 0}$ denote the FA-1f process restricted to $\Lambda=[-K-\overline{v}t,K+\overline{v}t]\cap\Z$ with empty boundary condition and started from $\sigma_\Lambda$. By finite speed of propagation, we have, for $f$ bounded with support in $[-K,K]$
 \begin{equation}\label{eq:relaxFSP}
|\E_\sigma[f(\sigma_t)]-\E_{\sigma_\Lambda}[f(\sigma^\Lambda_t)]|\leq 4\|f\|_\infty e^{-t}.
\end{equation}

Also, Proposition 5.1 in \cite{BCMRT13} ensures that for some $c=c(q)<\infty$,
\begin{multline*}
|\E_{\sigma_\Lambda}[f(\sigma^\Lambda_t)]|\leq c\|f\|_\infty\left(ne^{-q|\Lambda_1|}+t|\Lambda|\sup_{s\in[0,t]}\P_{\sigma_\Lambda}(\sigma^\Lambda_s\notin\mathcal{A})\right.\\
+\left.|\Lambda|\exp(-t/3)+\exp\left(-t/c+c|\Lambda|e^{-\frac{t}{c|\Lambda_1|}}\right)\vphantom{\frac{|\Lambda|}{\varepsilon t}}\right),
\end{multline*}
where $\mathcal{A}=\bigcap_{i=1}^n\left\{\sigma\in\lbrace 0,1\rbrace ^\Lambda\ \colon\ \sum_{x\in\Lambda_i}(1-\sigma(x))\geq 2\right\}$ is the event that there are at least two zeros in each $\Lambda_i$.
%Let us choose $\varepsilon=\varepsilon(K,q)$ such that $100 ce^{-1/c\varepsilon}\leq 1/2c$. This ensures that all the exponential terms in the above formula converge. 

It remains to control $\sup_{s\in[0,t]}\P_{\sigma_\Lambda}(\sigma^\Lambda_s\notin\mathcal{A})$. Here we need to recall another result from \cite{BCMRT13}. For $x\in\Lambda$, $\sigma\in\lbrace 0,1\rbrace^\Lambda$, let $\xi^x(\sigma)$ be the minimal distance from $x$ to an empty site in $\sigma$:
\begin{equation*}
\xi^x(\sigma)=\min\lbrace |x-y|\ \colon\ y\in\Lambda\cup\partial\Lambda\text{ s.t. }\sigma(y)=0\rbrace,
\end{equation*}
with the convention that $\sigma(y)=0$ for $y\in\partial\Lambda$ (recall that $\sigma^\Lambda$ is defined as the FA-1f process in $\Lambda$ with empty boundary conditions). Thanks to a comparison of $\xi^x$ with a random walk drifted towards $0$, Proposition 4.1 in \cite{BCMRT13} establishes that for all $x\in\Lambda$, $\theta\geq 1$ such that $q\in (\theta/(\theta+1),1]$, $\sigma\in\lbrace 0,1\rbrace^\Lambda$, $t\geq 0$,
\begin{equation}\label{eq:driftxi}
\E_\sigma\left[\theta^{\xi^x(\sigma^\Lambda_t)}\right]\leq \theta^{\xi^x(\sigma)}e^{-\lambda t}+\frac{q}{q(\theta+1)-\theta},
\end{equation}
where $\lambda=\frac{\theta^2-1}{\theta}\left(q-\frac{\theta}{\theta+1}\right)>0$.

Partition each $\Lambda_i$ into two intervals $\Lambda_i^+,\Lambda_i^-$ of length at least $|\Lambda_1|/2$, centered respectively in $x_i^+,x_i^-$. If $\sigma\notin\mathcal{A}$, necessarily there exists $i\in\{1,\ldots,n\}$ such that either $\sigma_{|\Lambda^+_i}\equiv 1$ or $\sigma_{|\Lambda^-_i}\equiv 1$. In particular, either $\xi^{x_i^+}(\sigma)$ or $\xi^{x_i^-}(\sigma)$ is larger than $|\Lambda_1|/4$. Therefore, for all $s\leq t$,
\begin{align*}
\P_{\sigma_\Lambda}(\sigma^\Lambda_s\notin\mathcal{A})& \leq \sum_{i=1}^n\left(\P_{\sigma_\Lambda}(\xi^{x_i^+}(\sigma^\Lambda_s)\geq |\Lambda_1|/4)+\P_{\sigma_\Lambda}(\xi^{x_i^-}(\sigma^\Lambda_s)\geq |\Lambda_1|/4)\right)\\
&\leq  \sum_{i=1}^n\theta^{-|\Lambda_1|/4}\left(\E_{\sigma_\Lambda}\left[\theta^{\xi^{x_i^+}(\sigma^\Lambda_s)}\right]+\E_{\sigma_\Lambda}\left[\theta^{\xi^{x_i^-}(\sigma^\Lambda_s)}\right]\right)\\
&\leq 2n\theta^{-|\Lambda_1|/4}\left(\theta^{|\Lambda_1|/8}+c\right)
\end{align*}
for some constant $c$ depending on $q,\theta$. The last inequality comes from \eqref{eq:driftxi} and the assumption on the maximal distance between two zeros in $\sigma$.
\end{proof}

\begin{proof}[Proof of Proposition~\ref{relaxation_finitevolume}]
The only change with respect to the above proof is that we don't need the finite volume propagation step \eqref{eq:relaxFSP}.
\end{proof}

%We now give a few applications of this result under different assumptions on the respective behaviors of $K$ and $t$.
\begin{corollary}\label{cor:relaxation}Let $q>1/2$ and $t>0$. Let $K>0$, $f$ a bounded function with support contained in $[-K,K]$ such that $\mu(f)=0$. If $K\leq e^{t^\alpha}$ with $\alpha\in (0,1/2)$, there exists $c'=c'(\alpha,q)>0$ such that,
\begin{itemize}
%\item If $K\leq\alpha t$ with $\alpha>0$, there exist $\beta=\beta(\alpha,q)>0$ and $c'=c'(\alpha,q)>0$ such that, if $\sigma\in\mathcal{H}(-K-\overline{v}t,K+\overline{v}t,\beta t)$,
%\begin{equation}
%|\E_\sigma[f(\sigma_t)]|\leq \frac{1}{c'}\|f\|_\infty e^{-c't}.
%\end{equation}
%\item If $K\leq t^\alpha$ with $\alpha>1$, there exist $\beta=\beta(\alpha,q)>0$ and $c'=c'(\alpha,q)>0$ such that, if $\sigma\in\mathcal{H}(-K-\overline{v}t,K+\overline{v}t,\beta t/\log t)$,
%\begin{equation}
%|\E_\sigma[f(\sigma_t)]|\leq \frac{1}{c'}\|f\|_\infty e^{-c'\frac{t}{\log t}}.
%\end{equation}
\item  if $\sigma\in\mathcal{H}(-K-\overline{v}t,K+\overline{v}t,\sqrt{t})$,
\begin{equation}\label{eq:relaxation}
|\E_\sigma[f(\sigma_t)]|\leq \frac{1}{c'}\|f\|_\infty e^{-c'\sqrt{t}};
\end{equation}
\item if $\sigma\in\mathcal{H}(-K,K,\sqrt{t})$,
\begin{equation}
|\E_\sigma[f(\sigma^{[-K,K]}_t)]|\leq \frac{1}{c'}\|f\|_\infty e^{-c'\sqrt{t}},\label{eq:relaxation_finitevolume}
\end{equation}
where $(\sigma^{[-K,K]}_t)_{t\geq 0}$ denotes the FA-1f process in ${[-K,K]}$ with zero boundary condition.
\end{itemize}
%The same holds for $|\E_\sigma[f(\sigma^\Lambda_t)]|$ under the hypotheses of Proposition~\ref{relaxation_finitevolume}.
\end{corollary}

The assumption $K\leq e^{t^\alpha}$ with $\alpha<1/2$ represents in fact the largest support we can consider for $f$ such that the estimate in Proposition~\ref{relaxation_lemma} is useful. Indeed, it does not give a vanishing estimate for $K=e^{\sqrt{t}}$.

\begin{proof}
It suffices to choose $|\Lambda_1|=8\sqrt{t}$ and apply Propositions~\ref{relaxation_lemma} and~\ref{relaxation_finitevolume}.%appropriately the size of $\Lambda_1$.
%\begin{itemize}
%\item If $K\leq\alpha t$ with $\alpha>0$, we let $|\Lambda_1|=4\beta t$ with $c(\alpha+100)e^{-(c\beta)^{-1}}=(2c)^{-1}$.
%\item If $K\leq t^\alpha$ with $\alpha>1$, we let $|\Lambda_1|=4\beta\frac{t}{\log t}$ with $\beta=(2c(\alpha-1))^{-1}$.
%\item If $K\leq e^{t^\alpha}$ with $\alpha\in (0,1/2)$, we let $|\Lambda_1|=4\sqrt{t}$.\qedhere
%\end{itemize}
\end{proof}
%\begin{corollary}\label{cor:relaxation_finitevolume}
%Let $K>0$, $f$ a bounded function with support contained in $\Lambda:=[-K,K]$ such that $\mu(f)=0$. If $K\leq e^{t^\alpha}$ with $\alpha\in (0,1/2)$ then there exists $c'=c'(\alpha,q)>0$ such that for all initial configuration $\sigma\in\mathcal{H}(-K,K,\sqrt{t})$,
%\[
%|\E_\sigma[f(\sigma^\Lambda_t)]|\leq \frac{1}{c'}\|f\|_\infty e^{-c'\sqrt{t}},
%\]
%where $(\sigma^\Lambda_t)_{t\geq 0}$ denotes the FA-1f process in $\Lambda$ with zero boundary condition.
%\end{corollary}

\section{Coupling between FA-1f and contact process}\label{section:couplingfapc}

We wish to exploit the comparison result between the FA-1f and contact processes (see Lemma~\ref{order}) to guarantee a sufficient number of zeros for the FA-1f dynamics. To this purpose, since the contact process can die, we will need first to do a restart argument.

%We want to compare the two processes at any time but the contact process can vanish so we will use a restart argument. %Let $\delta^x\in\Omega$ be defined by $\delta^x(x)=0$ and $\delta^x(y)=1$ if $y\neq x$.

\subsection{Restart argument}

\begin{lemma}\label{couplage1} Let $q>\overline{q}$. For any $\sigma\in LO_0$, there exist a process $(\sigma_t,\eta_t)_{t\geq 0}$ taking values in $\Omega^2$ and two random variables $T$ and $Y$ taking respectively their values in $\R^+$ and $\Z$ such that 
\begin{enumerate}
\item $(\sigma_t)_{t>0}$ is an FA-1f process starting from $\sigma\in LO_0$,
\item $\forall t>0,\forall x\in\Z, \sigma_t(x)\leq \eta_t(x)$,
\item $\left(\eta_{T+t}(Y+\cdot)\right)_{t>0}$ is a surviving threshold contact process starting from $\delta^0$. 
\end{enumerate}
Moreover, $T$ and $|Y|$ have exponentially decaying tail probabilities.
\end{lemma}

\begin{proof} The idea is to couple a FA-1f process with a contact process and to restart the second one each time that it vanishes. Eventually, the contact process will survive (because $q> \bar q$) and the space-time point $(Y,T)$ corresponding to the origin of this surviving contact process is not very far from the origin. The procedure is illustrated by Figure~\ref{fig:couplingfapc}.

\begin{figure}[bt]
{\centering
\begin{tikzpicture}[scale=0.7]

\draw[->] (-8,-4)--(-7,-7);
\draw (-8,-4) node[above]{Front FA};

%les courbes
\draw[-] (0,0) to[out=-120,in=145] (-2,-1.5) to[out=-35,in=65] (-3,-3);
 \draw[-] (-3,-3) to[out=-115,in=90]  (-4,-4) to[out=-90,in=130] (-5,-6) to[out=-50,in=90]  (-6,-7);
\draw[-] (-6,-7) to[out=-90,in=65] (-8,-8.5) to[out=-115,in=65] (-10,-10);
\draw[dotted] (-10,-10) to[out=-115,in=45] (-11,-11);

\draw (0,0)--(8,0);
\draw[dashed] (-10,-10)--(8,-10) node[right]{$t$};

%Vanishing PC 1
\draw (0,0) to[out=-120,in=125] (-1,-3);
\draw (0,0) to[out=-70,in=70] (-1,-3);
\fill[opacity=0.5,gray] (0,0) to[out=-120,in=125] (-1,-3) to[out=70,in=-70] (0,0);
\draw[<->](8.5,0)--(8.5,-3) node[midway, left]{$U_1$};
\draw[->] (2,-4) --(0.5,-2) ;
\draw (2,-4) node[right]{vanishing contact};
\draw (3,-4.8) node[right]{processes};

%Vanishing PC 2
 \draw (-3,-3) to[out=-115,in=140]    (-2.5,-7); %(-4,-4) to[out=-90,in=130] (-5,-6) to[out=-50,in=90]
  \draw (-3,-3) to[out=-70,in=70]  (-2.5,-7); %(-4,-4) to[out=-90,in=130] (-5,-6) to[out=-50,in=90]
\fill[opacity=0.5,gray] (-3,-3) to[out=-115,in=140]    (-2.5,-7) to[out=70,in=-70] (-3,-3);
\draw[<->](8.5,-3)--(8.5,-7) node[midway, left]{$U_2$};
\draw[->] (2,-4) --(-1.5,-5) ;
\foreach \x in { (-2.5,-7)}
 {\draw[thick] \x circle(0.2);
 \fill[white] \x circle(0.2);}

%surviving PC
  \draw (-6,-7) to[out=-90,in=65]  (-8,-10); %(-8,-8.5) to[out=-115,in=65]
  \draw[dotted] (-8,-10) to[out=-115,in=45] (-9,-11);
    \draw (-6,-7) to[out=-90,in=65]  (-5,-10); %(-8,-8.5) to[out=-115,in=65]
\draw[dotted] (-5,-10) to[out=-115,in=45] (-4,-11);

\fill[opacity=0.5,gray] (-6,-7) to[out=-90,in=65]  (-8,-10) --(-5,-10) to[out=65,in=-90] (-6,-7);

\draw[->] (-3,-8.8) --(-5,-8.8);
\draw (-3,-8.8) node[right]{surviving contact process};

\draw[<->](9,0)--(9,-7) node[midway, right]{$T$};
\draw[dashed](-6,0)--(-6,-7);
\draw[<->] (-6,0)--(-0.25,0) node[near start, above]{$Y$};
\draw[<->] (-3,0.25)--(-0.25,0.25) node[midway, above]{$X_1$};
\draw[dashed,gray](-3,0.25)--(-3,-3);
\draw[<->] (-0.75,-3)--(0,-3) node[midway, below,scale=0.7]{$Z_1$};
\draw[dashed,gray](0,0)--(0,-3);

  \draw[thick] (-8,-10) circle(0.2) node[below right]{$X(\eta_t)$};
 \fill[white] (-8,-10) circle(0.2);

 \foreach \x in { (0,0),(-3,-3), (-6,-7)}
 {\draw[thick] \x circle(0.2);
 \fill[white] \x circle(0.2);}
 
  \draw[thick] (-10,-10) circle(0.2) node[above left]{$X(\sigma_t)$};
 \fill[white] (-10,-10) circle(0.2);
 
 \foreach \x in { (-1,-3)}
 {\draw[thick] \x circle(0.2);
 \fill[white] \x circle(0.2);}

\end{tikzpicture}
\par}
\caption{Restart procedure, coupling between FA-1f and surviving contact process.}
\label{fig:couplingfapc}
\end{figure}
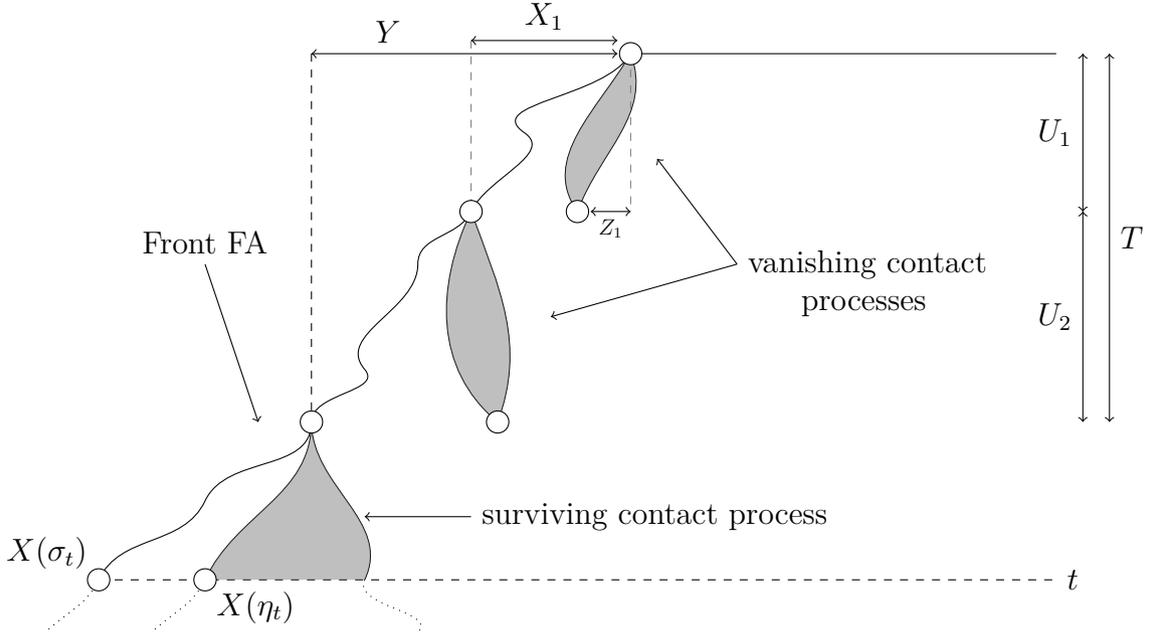

Let $(\mathcal{C}^{(i)})_{i\in\N}$ be a sequence of independent copies of the collection described in Section~\ref{graphical} and $\P$ their distribution. For $i\in\N$, let $\eta^{(i)}_\cdot=(\eta^{(i)}_t)_{t\geq 0}$ be the contact process started from $\delta^{0}$ constructed with $\mathcal{C}^{(i)}$. Further let $U_i=\tau (\eta^{(i)}_{\cdot})$ be the extinction time of $\eta^{(i)}_\cdot$ as defined in~\eqref{extinctiontimePC}. The random variables $(U_i)_{i\in\N}$ are independent and identically distributed and, by our choice of $q$, we have $\P(U_1=\infty)>0$. Let \[L=\min\{i\in\N, U_i=\infty\}.\] The random variable $L$ has geometric distribution. Moreover, conditionally on $\{L=l\}$, $(U_1,\ldots,U_{l-1})$ are i.i.d.\@ with the same distribution as $U_1$ conditioned on $U_1<\infty$. Let 
\[
T=\sum_{i=1}^{L-1} U_i,
\]
with $T=0$ if $L=1$.
Then $T$ has exponentially decaying tail probabilities. Indeed, from Estimate~\eqref{small_cluster} (\textit{cf.}\@ Appendix~\ref{annexe_cp}), we have for $t>0$ 
\begin{align*}\label{decexpoUi}
\P(t<U_i<\infty)=\P(t<\tau^{\{0\}}<\infty)\leq C_1\exp(-C_2t).
\end{align*} 
So, we can choose $\beta_1$ and $\beta_2$ such that $\E[e^{\beta_1 (L-1)}]<\infty$ and $\E\left[e^{\beta_2 U_1} \mid U_1<\infty\right]<e^{\beta_1}$. We have that
\begin{align*}
\E\left[e^{\beta_2 T}\right]=\E\left[\E\left[e^{\beta_2\sum_{i=1}^{L-1} U_i} |L\right]\right]=\E\left[\E\left[e^{\beta_2 U_1} \mid U_1<\infty\right]^{L-1}\right]\leq \E\left[e^{\beta_1 (L-1)}\right]<\infty.
\end{align*}

We construct recursively a sequence of processes $(\sigma^{[i]}_t,\eta^{[i]}_t)_{t\geq 0}$ and random variables $X_i\in\Z$ for $i\in\N$.
\begin{enumerate}
\item Let $(\sigma^{[1]}_t,\eta^{[1]}_t)_{t\geq 0}$ be the basic coupling started from $(\sigma,\delta^0)$ using $\mathcal{C}^{(1)}$. We define $X_1:=X(\sigma_{U_1})$ if $U_1<\infty$ and $0$ else.
\item Assuming $(\sigma^{[i]}_t,\eta^{[i]}_t)_{t\geq 0}$ and $X_i$ have been constructed, we define $(\sigma^{[i+1]}_t,\eta^{[i+1]}_t)_{t\geq 0}$ as follows.
\begin{itemize}
\item If $T_i:=\sum_{j=1}^iU_j=\infty$ then $(\sigma^{[i+1]}_t,\eta^{[i+1]}_t)_{t\geq 0}:=(\sigma^{[i]}_t,\eta^{[i]}_t)_{t\geq 0}$ and $X_{i+1}:=X_i$.
\item Else, let $(\sigma^{[i+1]}_t,\eta^{[i+1]}_t)=(\sigma^{[i]}_t,\eta^{[i]}_t)$ for $t<T_i$ and $(\sigma^{[i+1]}_{T_i+t},\eta^{[i+1]}_{T_i+t})_{t\geq 0}$ be the basic coupling started from $(\sigma^{[i]}_{T_i},\delta^{X_i})$ using $\theta_{X_i}\mathcal{C}^{(i+1)}$ the spatial translation by $X_i$ of the collection $\mathcal{C}^{(i+1)}$ (in particular $U_{i+1}$ is the extinction time of $(\eta^{[i+1]}_{T_i+t})_{t\geq 0}$). We choose $X_{i+1}:=X(\sigma^{[i+1]}_{T_i+U_{i+1}})$ if $U_{i+1}<\infty$ and $X_{i+1}:=X_i$ else.
\end{itemize}
\end{enumerate}

Since $L$ has geometric distribution, the algorithm fixates almost surely in finite time: for $i\geq L$, $(\sigma^{[i+1]}_t,\eta^{[i+1]}_t)_{t\geq 0}=(\sigma^{[i]}_t,\eta^{[i]}_t)_{t\geq 0}$. This allows to define $(\sigma_t,\eta_t)_{t\geq 0}$ as  $(\sigma^{[L]}_t,\eta^{[L]}_t)_{t\geq 0}$ and $Y=X_{L-1}$. Moreover, since the $U_i$ are stopping times, $(\sigma_t)_{t\geq 0}$ is a FA-1f process started from $\sigma$. We also have immediately that $\left(\eta_{T+t}(Y+\cdot)\right)_{t>0}$ is a surviving threshold contact process starting from $\delta^0$. Finally, Lemma~\ref{order} implies that $\sigma_t\leq\eta_t$ for all $t\geq 0$; indeed, by definition, $X_i$ is a zero of the configuration $\sigma^{[i]}_{T_i}$ and therefore $\sigma^{[i]}_{T_i}\leq\delta^{X_i}$ for $i\leq L-1$.

It remains to show that $Y$ has exponentially decaying tail probability. To that end, for $i\leq L-1$ let $Z_i$ be the position of the unique zero in $\eta_{T_i^-}$. We have that $X_{i}\leq Z_i+1$. Indeed, by definition of $Z_i$, $\eta_{U_i^-}(Z_i)=\sigma_{U_i^-}(Z_i)=0$. If $\sigma_{U_i^-}(Z_i-1)=\sigma_{U_i^-}(Z_i+1)=1$ then $\sigma_{U_i}(Z_i)=0$ because the FA-1f constraint is not satisfied and the zero at position $Z_i$ at time $U_i^-$ can not update. If $\sigma_{U_i^-}(Z_i-1)$ or $\sigma_{U_i^-}(Z_i+1)$ is equal to $0$, it is still equal to $0$ at time $U_i$ because there is already a Poisson clock ringing at position $Z_i$ at time $U_i$, so $X_{i}\leq Z_i+1$. Thus,
\[
Y= X_{L-1}\leq \sum_{i=1}^{L-1} (Z_i-X_{i-1}) +L,
\] 
with $X_0=0$.
Conditionally on the event $\{L=l\}$, the random variables $\{Z_i-X_{i-1}, i=1,\ldots, l-1\}$ are independent and have the same distribution: they represent the expansion of a non-surviving contact process. For $t>0$,
\begin{align*}
\P(T\leq at, |Y|>t,L\leq t/2)
%&\aureliacomments{\leq \P\left(\sum_{i=1}^{L-1} U_i\leq at,\left|\sum_{i=1}^{L-1} Z_i-X_{i-1}\right|+L>t, L\leq t/2\right)}\\
&\leq \P\left(\sum_{i=1}^{L-1} U_i\leq at,\left|\sum_{i=1}^{L-1} Z_i-X_{i-1}\right|>t/2\right)\\
&\leq \sum_{l=1}^\infty\P(L=l)\P\left(\sum_{i=1}^{l-1} U_i\leq at,\left|\sum_{i=1}^{l-1} Z_i-X_{i-1}\right|>t/2\big|L=l\right)\\
&\leq A\exp(-Bt)
\end{align*}
using the `at most linearity'~\eqref{atmostlinearpc} of the contact process with $a<\frac{1}{2\overline{v}}$. 
We conclude using the exponentially decaying tail of $L$ and $T$:
\begin{align*}
\P\left(|Y|>t\right)&\leq \P(T>at)+\P(L> t/2) +\P(T\leq at, |Y|>t,L\leq t/2)
\\
&\leq A\exp(-Bt).\qedhere
\end{align*}
\end{proof}

\subsection{Consequences of the coupling}
The first consequence of the previous coupling is the 'at least linear growth' of the front of the FA-1f process. 
\begin{corollary}\label{atleastlinearFA} Let $q>\overline{q}$. There exists $\underline{v}>0$ and $A,B>0$ such that for every $\sigma\in LO_0$ and $t>0$,
\[\P(X(\sigma_t)>-\underline{v}t)\leq A\exp(-Bt).\]
\end{corollary}

\begin{proof} Denote by $v_{cp}$ the velocity of the contact process (see Theorem~\ref{DGestimates}). Choose $\underline{v}=v_{cp}/2$. Let $(\sigma_t,\eta_t)_{t\geq 0}$, $T$ and $Y$ be the objects defined in Lemma~\ref{couplage1}. We denote by $X(\eta_t)$ the position of the leftmost zero at time $t$ in a contact process started from $\delta^0$.
For every $t>0$, $c\in(0,1)$,
\begin{align*}
\P~\big(X(\sigma_t)>-\underline{v}t&, 0\leq T\leq ct, Y\leq ct\big)\\
&\leq \P\left(X(\eta_t)-Y>-(\underline{v}+c)t, (1-c)t\leq t-T\leq t\right)\\
&\leq  \P\left(X(\eta_{T+(t-T)}(Y+\cdot))>-\frac{\underline{v}+c}{1-c}(t-T), (1-c)t\leq t-T\leq t\right) 
\\
&\leq \sup_{u\in[(1-c)t,t]} \P\left(X(\eta_{T+u}(Y+\cdot))>-\frac{\underline{v}+c}{1-c}u\right)
\\
&\leq \sup_{u\in[(1-c)t,t]} A\exp(-B u)=A\exp(-B't)
% &\phantom{aaaaaaa}\leq \int_{u=(1-c)t}^t \P(X(\eta_{T+u}(Y+\cdot))>-v'u)d\P(T=t-u)\\
% &\phantom{aaaaaaa}\leq \int_{u=(1-c)t}^t A\exp(-Bu)d\P(T=t-u) \leq A'\exp(-B't)
\end{align*}
where we used that $\sigma_t\leq\eta_t$ (and therefore $X(\sigma_t)\leq X(\eta_t)$) in the first line.  We choose $c>0$ such that $\frac{\underline{v}+c}{1-c}< v_{cp}$ and apply~\eqref{atleastlinear} of Theorem~\ref{DGestimates} to the surviving contact process $\tilde{\eta}_u(\cdot)=\eta_{T+u}(Y+\cdot)$ and we bound the last probability. The fact that $T$ and $|Y|$ have exponentially decaying tails allows to conclude. 
\end{proof}

The second consequence of the coupling is to guarantee a minimal quantity of zeros around the origin in the FA-1f process.
\begin{corollary}\label{cor:zerozone}Let $q>\overline{q}$. There exist $c_1,c_2>0$ such that for any $\sigma\in LO_0$, $t>0$, if $\sigma(x)=0$ %and $l\gg \log t$ 
then we have
\[\P_{\sigma}(\sigma_t\notin\mathcal{H}(x-\underline{v} t ,x+\underline{v} t,l))\leq c_1 t\exp\left(-c_2 \left(t\wedge l\right)\right).\]
\end{corollary}

\begin{proof} Let us do the proof for $x=0$. We use again the coupling from Lemma~\ref{couplage1}. For $t>0$, $\alpha\leq v_{cp}-\underline{v}$,
\begin{align*}
\P_{\sigma}&\left(\sigma_t\notin\mathcal{H}(-\underline{v} t ,\underline{v} t,l),T\leq t/2,|Y|\leq \alpha t\right)\\
%&\phantom{aaaa}\aureliacomments{\leq\P\left(\exists I_l\subset [-\underline{v} t ,+\underline{v} t ],~ |I_l|=l\text{ and }\eta_t(I_l)\equiv 1,  T\leq t/2,|Y-x|\leq \alpha t/2\right)}\\
&\phantom{aaaa}\leq\P\left(\exists y\in [Y-(\underline{v}+\alpha)t ,Y+(\underline{v}+\alpha)t-l+1]\text{ s.t. }\eta_t([y,y+l-1])\equiv 1,  T\leq t/2\right)\\
%&\phantom{aaaa}\aureliacomments{\leq\P\left(\exists I_l\subset [Y-\underline{v} U ,Y+\underline{v} U ],~ |I_l|=l\text{ and }\eta_{T+U}(I_l)\equiv 1,  U\geq t/2\right)}\\
&\phantom{aaaa}\leq \sum_{y=-(\underline{v}+\alpha)t}^{(\underline{v}+\alpha)t} \P\left(\forall z\in [y,y+l-1],\eta_{T+(t-T)}(Y+z)=1, t/2\leq t-T\leq t\right)\\
&\phantom{aaaa}\leq 2{v}_{cp} t \int_{u=t/2}^t A\exp(-B(u\wedge l))d\P(T=u-t),\\
&\phantom{aaaa}\leq A't \exp(-B'(t\wedge l)),
\end{align*}
where we applied the Corollary~\ref{zerosCP} to the surviving contact process $\tilde{\eta}_u(\cdot)=\eta_{T+u}(Y+\cdot)$ with $\underline{v}\leq v_{cp}$. The fact that $T$ and $|Y|$ have exponentially decaying tails allows to conclude. If $x\neq 0$ we can do the same proof but we have to start the coupling from the point $x$ and to restart from the closest zero (instead of the front) when the contact process dies.
\end{proof}

\subsection{Zeros lemma}

In the following lemma we use repeatedly Corollary~\ref{cor:zerozone} to control the probability that, at time~$s$, at distance~$L$ from the front and over a distance~$M$, we have no 0-gap larger than~$l$. 

\begin{lemma}\label{zeros_lemma} Let $q>\overline{q}$. Let $s,l,M,L>0$ and $\sigma\in LO_0$.
\begin{enumerate}
\item If $L+M\leq 2\underline{v} s$ then there exists $c>0$ depending only on $p$ such that 
\[
\P_\sigma\left(\theta_{X(s)}\sigma_s\notin \mathcal{H}(L, L+M, l)\right)\leq  {(L+M)^2}\exp \left(-c\left(L \wedge l\right) \right).
\]
\item If $L+M \geq 2\underline{v} s$ and $\sigma\in \mathcal{H}(0, L+M, 2\underline{v} s )$ then there exists $c>0$ such that 
\[
\P_\sigma\left(\theta_{X(s)}\sigma_s\notin \mathcal{H}(L, L+M, l)\right)\leq \frac{ {s^2}}{L}\exp \left(-c\left( L\wedge l\right) \right)+M {s}\exp(-c( s\wedge l)).\]
\end{enumerate}
\end{lemma}

\begin{proof} The strategy is the following: we consider a number of zeros which we know are present in the dynamics, either because they are present in the initial configuration or because they are well-chosen intermediate positions occupied by the front. For each of these zeros, we use Corollary~\ref{cor:zerozone} to guarantee that at time $s$, a given interval around them contains no gap larger than $l/2$. We then control that w.h.p.\@ the different intervals thus obtained cover $[X(s)+L,X(s)+L+M]$, and therefore $\theta_{X(s)}\sigma_s\in \mathcal{H}(L, L+M, l)$. The technique is illustrated by Figure~\ref{fig:zeroslemma}.

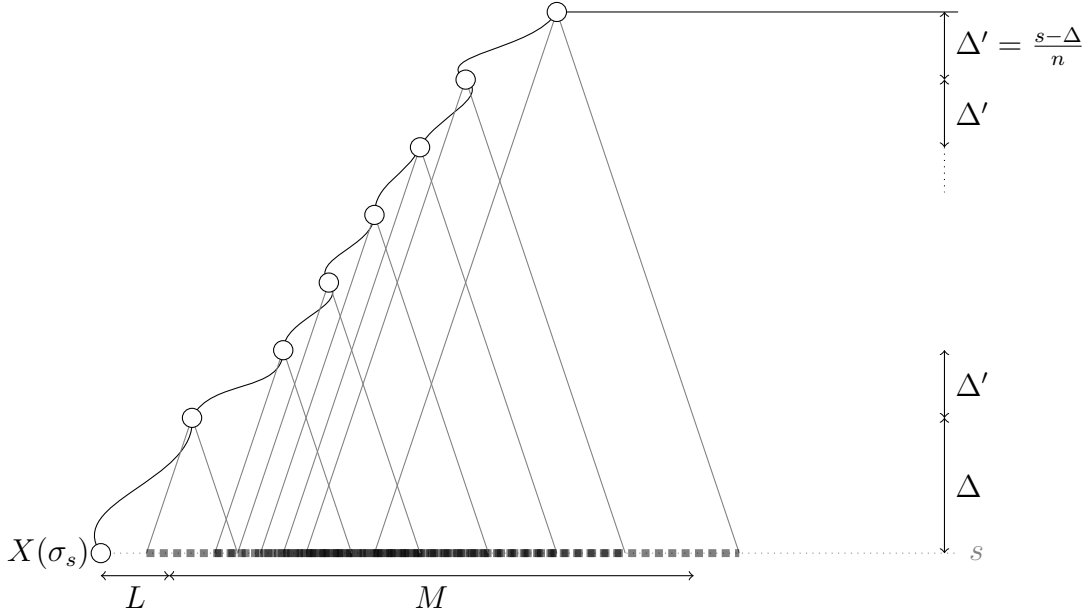
\begin{figure}[bt]
{\centering
\begin{tikzpicture}[scale=0.6]

%les courbes
\draw[-] (0,0) to[out=-120,in=145] (-2,-1.5) to[out=-35,in=65] (-3,-3) to[out=-115,in=90]  (-4,-4.5) to[out=-90,in=130] (-5,-6) to[out=-50,in=90]  (-6,-7.5) to[out=-90,in=65] (-8,-9) to[out=-85,in=125] (-10,-12);

%Les lignes et les temps
\draw (0,0)--(8.8,0);
\draw[dotted,gray] (-10,-12)--(8.8,-12) node[right]{$s$};
\draw[<->](8.5,0)--(8.5,-1.5) node[midway, right]{$\Delta'=\frac{s-\Delta}{n}$};
\draw[<->](8.5,-1.5)--(8.5,-3) node[midway, right]{$\Delta'$};
\draw[<->](8.5,-7.5)--(8.5,-9) node[midway, right]{$\Delta'$};
\draw[dotted](8.5,-3)--(8.5,-4);
\draw[<->](8.5,-9)--(8.5,-12) node[midway, right]{$\Delta$};

  \draw(-10,-12)  node[left]{$X(\sigma_s)$};
  \draw[<->] (-10,-12.5)--(-8.5,-12.5) node[below, midway]{$L$};
    \draw[<->] (-8.5,-12.5)--(3,-12.5) node[below, midway]{$M$};

%LES FRONTS
\foreach \pt in { (-10,-12)}  { 
  \draw[thick] \pt circle(0.2); 
  \fill[white] \pt circle(0.2);
}

\foreach \pt in { (0,0)} {
  \parsept{\x}{\y}{\pt}
	\draw[gray] \pt -- (\x-4-\y/3, -12);
	\draw[gray] \pt -- (\x+4+\y/3, -12);
	\draw[thick] \pt circle(0.2); \fill[white] \pt circle(0.2);
	\draw[dotted, line width=3pt, opacity=0.5] (\x-4-\y/3, -12)--(\x+4+\y/3, -12);
}

\foreach \pt in { (-2,-1.5)} {
  \parsept{\x}{\y}{\pt}
	\draw[gray] \pt -- (\x-4-\y/3, -12);
	\draw[gray] \pt -- (\x+4+\y/3, -12);
	\draw[thick] \pt circle(0.2); \fill[white] \pt circle(0.2);
	\draw[dotted, line width=3pt, opacity=0.5] (\x-4-\y/3, -12)--(\x+4+\y/3, -12);
}

\foreach \pt in {(-3,-3), (-4,-4.5), (-5,-6),(-6,-7.5)} {
  \parsept{\x}{\y}{\pt}
	\draw[gray] \pt -- (\x-4-\y/3, -12);
	\draw[gray] \pt -- (\x+4+\y/3, -12);
	\draw[thick] \pt circle(0.2); \fill[white] \pt circle(0.2);
	\draw[dotted, line width=3pt, opacity=0.5] (\x-4-\y/3, -12)--(\x+4+\y/3, -12);
}

\foreach \pt in { (-8,-9)} {
  \parsept{\x}{\y}{\pt}
	\draw[gray] \pt -- (\x-4-\y/3, -12);
	\draw[gray] \pt -- (\x+4+\y/3, -12);
	\draw[dotted, line width=3pt, opacity=0.5] (\x-4-\y/3, -12)--(\x+4+\y/3, -12);
	\draw[thick] \pt circle(0.2); \fill[white] \pt circle(0.2);
}

\end{tikzpicture}

\par}
\caption{There is no 0-gap bigger than $\ell$ on the interval $[L,L+M]$ seen from the front.}
\label{fig:zeroslemma}
\end{figure}

Let us define the intermediate times we consider. Let
\begin{eqnarray*}
\Delta&=& \frac{L}{\overline{v}-\underline{v}}\wedge s,\\
n&=&\left\lceil\frac{(s-\Delta)(\overline{v}-\underline{v})}{2\underline{v}\Delta }\right\rceil,\\
\Delta'&=&\frac{s-\Delta}{n},\\
s_i&=&i\Delta'\text{ for }i\in\{0,\ldots,n\}.
\end{eqnarray*}
For $i\in\{0,\ldots,n\}$, by Corollary~\ref{cor:zerozone} and the Markov property applied at time $s_i$, for some constants $C,c>0$
\begin{equation*}
\P_\sigma\left(\sigma_s\notin \mathcal{H}(X(s_i)-\underline{v}(s-s_i), X(s_i)+\underline{v}(s-s_i), l/2)\right)\leq C(s-s_i)e^{-c(s-s_i)\wedge l},
\end{equation*}
\emph{i.e.\@} $[X(s_i)-\underline{v}(s-s_i),X(s_i)+\underline{v}(s-s_i)]$ contains no gap larger than $l/2$ with high probability.

In case 2, we also need to use the zeros of the initial configuration. Let $0=:x_0< x_1 < \ldots< x_m$ be the ordered set of zeros located between $0$ and $L+M$ in the initial configuration $\sigma$. Then by Corollary~\ref{cor:zerozone}, for $i\in\{0,\ldots,m\}$,
\begin{equation*}
\P_\sigma\left(\sigma_s\notin \mathcal{H}(x_i-\underline{v}s, x_i+\underline{v}s, l/2)\right)\leq Cs e^{-c(s\wedge l)},
\end{equation*}
\emph{i.e.\@} $[x_i-\underline{v}s,x_i+\underline{v}s]$ contains no gap larger than $l/2$ with high probability.

The next step is to control the respective positions of the intervals we introduced to check that with high probability they cover $[X(s)+L,X(s)+L+M]$. Let us consider for all $i\in\{0,\ldots,n-1\}$ the events that 
\begin{eqnarray}
\underline{v}\Delta\leq &X(s_n)-X(s)&\leq \overline{v}\Delta\label{eq:trajfront1}\\
\underline{v}\Delta'\leq &X(s_{i})-X(s_{i+1})&\leq \overline{v}\Delta'.\label{eq:trajfront2}
%0\leq &X(s_{n-k})-X(s)&\geq \underline{v}(s-s_{n-k})= \underline{v}\left(\Delta+k\frac{s-\Delta}{n}\right)
\end{eqnarray}

Fix $k\in\{0,\ldots,n\}$. With our choice of $\Delta$ and $n$, if \eqref{eq:trajfront1} and \eqref{eq:trajfront2} hold for $i\in\{n-k,\ldots,n-1\}$, we have for $i\in\{n-k,\ldots,n-1\}$
\begin{align}
X(s_n)-\underline{v}\Delta &\leq X(s)+L,\\
\label{tintin}  X(s_{i})-\underline{v}(\Delta+(n-i)\Delta')&\leq X(s_{i+1})+\underline{v}(\Delta+(n-i-1)\Delta'),\\
X(s)+2\underline{v}\left(\Delta+k\Delta'\right)&\leq X(s_{n-k})+\underline{v}(s-s_{n-k}).
\end{align}
To derive~\eqref{tintin} we used that $\overline{v}\Delta'\leq \underline{v}(2\Delta+\Delta')$. 
These equations in turn imply that $\displaystyle \bigcup_{i=n-k}^n\left[X(s_i)-\underline{v}(s-s_i), X(s_i)+\underline{v}(s-s_i)\right]$ is a covering of $[X(s)+L,X(s)+2\underline{v}(\Delta+k\Delta')]$. 
Similarly, if $\sigma\in\mathcal{H}(0, L+M, 2\underline{v} s)$ then we know that $x_{i+1}-x_i\leq 2\underline{v} s $ for $i\in\{0,\ldots,m-1\}$ and $x_m\geq L+M-2\underline{v}s$. 
Consequently, if \eqref{eq:trajfront1} and \eqref{eq:trajfront2} hold for all $i\in\{0,\ldots,n\}$ then we have
\[
\bigcup_{i=0}^n\left[X(s_i)-\underline{v}(s-s_i), X(s_i)+\underline{v}(s-s_i)\right]\cup\bigcup_{j=1}^m[x_j-\overline{v}s,x_j+\overline{v}s]\supset[X(s)+L,L+M-\underline{v}s].
\]

\noindent\underline{Case 1:} $L+M\leq 2\underline{v}s$\\
Choose $k=\lceil \frac{L+M-2\underline{v}\Delta}{2\underline{v}\Delta'}\rceil$. Then $2\underline{v}(\Delta+k\Delta')\geq L+M$, and the above arguments and the bounds we have on the speed of the front yield
\begin{align*}
\P_\sigma\big(\theta_{X(s)}\sigma_s\notin& ~\mathcal{H}(L, L+M, l)\big) \\
\leq &\ \P_\sigma\left(\sigma_s\in\bigcup_{i=n-k}^n \mathcal{H}(X(s_i)-\underline{v}(s-s_i),X(s_i)+\underline{v}(s-s_i),l/2)^c\right)\\
&+\ \P_\sigma\left(\text{\eqref{eq:trajfront1} is not satisfied }\right)\\
&+\ \P_\sigma\left( \exists i\in\{n-k,\ldots,n-1\}\text{ for which \eqref{eq:trajfront2} is not satisfied}\right)\\
\leq  &\ k(k\Delta'+\Delta)c_1\exp \left(-c_2\left(\left(\Delta'+\Delta\right)\wedge l\right)\right)\\
&\quad+A\exp\left(-B \Delta\right)+kA\exp\left(-B \Delta'\right)\\
\leq &\ c_1'(L+M)^2\exp \left(-c_2'\left(L\wedge l\right) \right),
\end{align*}
where we used Corollary~\ref{atleastlinearFA} and Lemma~\ref{finitespeedpropagation}.\\

\noindent\underline{Case 2:} $L+M\geq 2\underline{v}s$\\
Our arguments imply that 
\begin{align*}
\P_\sigma\big(\theta_{X(s)}\sigma_s\notin& ~\mathcal{H}(L, L+M, l)\big)\\
 \leq &\ \P_\sigma\left(\sigma_s\in\bigcup_{i=0}^n \mathcal{H}(X(s_i)-\underline{v}(s-s_i),X(s_i)+\underline{v}(s-s_i),l/2)^c\right)\\
&+\ \P_\sigma\left(\sigma_s\in\bigcup_{i=1}^m \mathcal{H}(x_i-\underline{v}s,x_i+\underline{v}s,l/2)^c\right)\\
&+\ \P_\sigma\left(\text{\eqref{eq:trajfront1} is not satisfied }\right)\\
&+\ \P_\sigma\left( \exists i\in\{0,\ldots,n-1\}\text{ for which \eqref{eq:trajfront2} is not satisfied}\right)\\
\leq  &\ c_1'\frac{s^2}{L}\exp \left(-c_2'\left(L\wedge l\right) \right)+ Ms\exp(-c( s\wedge l)).\qedhere
\end{align*}
\end{proof}

\section{Relaxation far from the front}\label{section:decorrelation}
Now, we will use the bounds on the speed of the front (\eqref{atmostlinearfa} and Corollary \ref{atleastlinearFA}) and relaxation results (Corollary~\ref{cor:relaxation}) to prove a relaxation far from the front. To do that we need to decorrelate the front trajectory from the interval in which we want to relax.
\begin{theorem}
\label{thm:decorrelation}
Let $q>\overline{q}$ and $\sigma\in LO_0$. Let $\alpha<1/2$ and $\delta>0$. There exists $c>0$ such that for any $M\leq e^{\delta t^\alpha}$, any $f$ with support in $[0,M]$, $\mu(f)=0$ and $\|f\|_\infty\leq 1$, if $\sigma\in\mathcal{H}\left(\underline{v}t,M+(4\overline{v}-\underline{v})t,\sqrt{t}\right)$, then
%Let $L\leq Ct$, $M\leq e^{L^\alpha}$ with $\alpha<1/2$. Then there exist $\delta, c>0$ such that if $\sigma\in\mathcal{H}\left(0,M,2\underline{v}\sqrt{t}\right)$, for any $f$ with support in $[0,M]$, $\mu(f)=0$ and $\|f\|_\infty\leq 1$,
%\noteO{
%The estimate would be better with a stronger hypothesis on $M$. For instance, $M\lesssim L$ would give exponential decay, in some cases without condition on the initial configuration. But I don't think we need this result for the coupling.}
\begin{equation*}
\left|\E_{\sigma}\left[f\left(\theta_{X(\sigma_t)+ {3\overline{v}t}}\sigma_t\right)\right]\right|\leq e^{-c\sqrt{ {t}}}.
\end{equation*}
\end{theorem}

\begin{remark}
 {The proof strategy of the equivalent theorem in \cite{front_east} combined with Lemma \ref{zeros_lemma} and Corollary~\ref{cor:relaxation} would also allow us to give a result of the type ``in the box $[L,L+M]$ seen from the front, the distribution is within $e^{-c\sqrt{L\wedge t}}$ of $\mu$ in total variation distance'', under suitable hypotheses on the initial configuration (depending on the respective regimes of $L,t,M$). A proper statement would however be too technical for the purpose of the present paper and we restrict to the above.} 
\end{remark}

\begin{proof}
% Let $s=\left(t-\frac{L}{3\overline{v}}\right)\vee 0$. 
Using Corollary~\ref{atleastlinearFA} and Equation~\eqref{atmostlinearfa} we can write that
\[
\E_{\sigma}\left[f\left(\theta_{X(\sigma_t)+ {3\overline{v}t}}\sigma_t\right)\right]= {\sum_{y=\underline{v}t}^{\overline{v}t} \E_{\sigma}\left[\1_{X(\sigma_t)=-y}f(\theta_{ {3\overline{v}t}-y}\sigma_{t})\right] +O\left(e^{-\gamma t}\right).}
\]
%where $\tilde{\sigma}_{t-s}$ denotes the configuration at time $t-s$ starting from $\sigma_s$.
By finite speed of propagation, we have with probability $1-O(e^{- {t}})$ 
\[
\max_{u\in[ {0},t]}|X(\sigma_u)|\leq  {\overline{v}t}.
\]
So, %since $L\geq 3\overline{v}(t-s)$ by construction, by Lemma~\ref{finitespeedpropagation}, 
 {for $y\leq \overline{v}t$}, the probability that there exists a sequence of successive clock rings between $ {3\overline{v}t}-y$ and $\displaystyle \max_{u\in[ {0},t]} |X(\sigma_u)|$ is less than $O(e^{- {t}})$. % if $y\in[\underline{v}(t-s),\overline{v}(t-s)]$. 
 On the event where there is no such sequence, $\1_{X(\sigma_t)=-y}$ and $f(\theta_{ {3\overline{v}t}-y}\sigma_{ {t}}$) are independent. Therefore,

\begin{multline}
\E_{\sigma}\left[f\left(\theta_{X(\sigma_t)+ {3\overline{v}t}}\sigma_t\right)\right]=\sum_{y=\underline {v} {t}}^{\overline{v} {t}}  {\P_{\sigma}\left(X(\sigma_t)=-y\right)\E_{\sigma}\left[f(\theta_{ {3\overline{v}t}-y}{\sigma}_{t})\right]} 
+O\left(e^{-\gamma {t}}+e^{- {t}}\right).
\end{multline}
In order to apply our relaxation result Corollary~\ref{cor:relaxation} on the interval 
\[\left[ {3\overline{v}t}-y,  {3\overline{v}t}-y+M\right],\]
 {we need to check that
\begin{align*}
3\overline{v}t-y-\overline{v}t&\geq \underline{v}t\\
3\overline{v}t-y+M+\overline{v}t&\leq M+(4\overline{v}-\underline{v})t,
\end{align*}
 which is clearly satisfied if $y\in[\underline{v}t,\overline{v}t]$. Therefore, our assumption on $\sigma$ and Corollary~\ref{cor:relaxation} imply that}
% let us consider the quantities
%\begin{align*}
%L'&=L-y-\overline{v}(t-s)\in\left[L/3,\left(\frac{2}{3}-\frac{\ \underline{v}\ }{\overline{v}}\right)L\right],\\
%M'&=M+2\overline{v}(t-s),\\
%l&=\sqrt{t-s},
%\end{align*}
%and the event representing the amount of zeros we need
%\[
%\mathcal{Z}=\left\{ \theta_{X(\sigma_s)}\sigma_s \in \mathcal{H}(L',L'+M',l)\right\}.
%\]
%Then, we can write
%\begin{multline}
%\E_{\sigma}\left[f\left(\theta_{X(\sigma_t)+ {3\overline{v}t}}\sigma_t\right)\right]=
%\sum_{y=\underline{v}t}^{\overline{v}t} \E_{\sigma}\left[\P_{\sigma_s}\left(X(\sigma_t)-X(\sigma_s)=-y\right)\1_{\mathcal{Z}}\E_{\sigma_s}\left[f(\theta_{X(\sigma_s)-y+L}\tilde{\sigma}_{t-s}\right]\right] \\
%+O\left(e^{-\gamma(t-s)}+e^{-(t-s)}+\P(\mathcal{Z}^c)\right).
%\end{multline}
%On the event $\mathcal{Z}$ the conditions required in \eqref{eq:relaxation} of Corollary~\ref{cor:relaxation} are satisfied and we have that
\[
 {\left|\E_{\sigma}\left[f(\theta_{ {3\overline{v}t}-y}{\sigma}_{t})\right]\right|\leq C e^{-c\sqrt{t}},}
\]
which in turn yields the desired result.
%Finally
%\[
%\E_{\sigma}\left[f(\theta_{X(\sigma_t)+L}\sigma_t)\right]=O\left(e^{-\gamma(t-s)}+e^{-(t-s)}+\P(\mathcal{Z}^c)+e^{-c'\sqrt{L}}\right).
%\]
%Now, we use Lemma~\ref{zeros_lemma} to control $\P(\mathcal{Z}^c)$. We have the three following cases:
%\begin{enumerate}
%\item if $L'+M'=L-y+M+\overline{v}(t-s)\leq 2\underline{v}s$, we are in the first case of Lemma~\ref{zeros_lemma} and 
%\[
%\P(\mathcal{Z}^c)\leq  {(L+M)^2}e^{-c(L\wedge l)}\leq e^{-c'\sqrt{L}};
%\]
%\item if $s\leq \sqrt{t}$ then $\sigma\in\mathcal{H}(0,L+M,2\underline{v}\sqrt{t})$ implies $\mathcal{Z}$ (using Corollary~\ref{cor:zerozone} for example);
%\item if $L'+M'=L-y+M+\overline{v}(t-s)\geq 2\underline{v}s $ and $s\geq\sqrt{t}$ then $\sigma\in\mathcal{H}(0,L+M,2\underline{v}\sqrt{t})$ implies that $\sigma\in\mathcal{H}(0,L+M,2\underline{v}s)$ and the second case of Lemma~\ref{zeros_lemma} yields that
%\[
%\P(\mathcal{Z}^c)\leq \frac{ {s^2}}{L}e^{-c(L\wedge l)}+M {s}e^{-c(s\wedge l)}\leq e^{-c'\sqrt{L}}.\qedhere
%\]
%\end{enumerate}
\end{proof}

\section{Invariant measure: proof of Theorem~\ref{thm:coupling}}\label{section:coupling}

\begin{proof}[Proof of Theorem~\ref{thm:coupling}] We start with two initial configurations $\sigma$ and $\sigma'$ in $LO_0$ and we prove that there exist $d^*>0$, $c>0$ (independent of $(\sigma, \sigma')$) such that for $t$ large enough
\begin{align}
\label{thmconfig}
\|\tilde{\mu}_t^{\sigma}-\tilde{\mu}_t^{\sigma'}\|_{[0,d^*t]}\leq e^{-ce^{(\log t)^{1/4}}},
\end{align}
where $\tilde{\mu}_t^{\sigma}$ is the distribution of the configuration seen from the front at time $t$, that is $\theta_{X(\sigma_t)}\sigma_t$.

Thanks to Theorem~\ref{thm:decorrelation} we know that far from the front the configurations starting respectively from $\sigma$ and $\sigma'$ will be close to a configuration sampled by the equilibrium measure, so they will be close tp one another (for the total variation distance). Following the strategy of~\cite{front_east,GLM15}, we construct a coupling where we use this property and where we wait until the configurations also coincide near to the front. Given $\epsilon>0$ and $t>0$, we fix the following quantities:

\begin{align*}
t_0&=(1-\epsilon )t\\
\Delta_1&=e^{(\log t)^{1/4}}\\
\Delta_2&=(\log t)^{3/4}\\
\Delta&=\Delta_1+\Delta_2.
\end{align*}

The 'perfect procedure' would be:
\begin{enumerate}
\item[Step 0.] Both configurations have a lot of zeros at time $t_0$
\item[Step 1.] Thanks to these zeros, they both closely match equilibrium far from the front after a time-lag $\Delta_1$. Thus, they match each other for the same conditions. 
\item[Step 2.] Then, during a time-lag $\Delta_2$, a very favorable event happens and the configurations coincide also near to the front.
\end{enumerate} 
Roughly speaking, the steps 0 and 1 are very likely and the step 2 has very small probability. So we will repeat step 2 a lot in order to make it succeed. In practice we also need to repeat step 1 because multiple tries of step 2 could destroy the assets of step 1. To do that, the time $t-t_0$ will be split in $N=\left\lfloor \frac{\epsilon t}{\Delta}\right\rfloor$ repetitions of steps 1 and 2, respectively of duration $\Delta_1$ and $\Delta_2$. For $n\in\{0,\ldots,N\}$, let $t_n=t_0+n\Delta$ (resp. $s_{n+1}=t_n+\Delta_1$) the instants at which each of the $N$ repetitions of step 1 (resp. step 2) begins. The repetition of the N steps is illustrated by Figure~\ref{fig:coupling} and the steps 1 and 2 are illustrated by Figure~\ref{fig:couplingzoom}.

\begin{figure}[p]
{\centering
\begin{tikzpicture}[scale=0.55]

%les courbes

\draw[-] (0,0) to[out=-120,in=145] (-2,-1.5) to[out=-35,in=65] (-3,-3) to[out=-115,in=90]  (-4,-4.5) to[out=-90,in=130] (-5,-6) to[out=-50,in=90]  (-6,-7.5) to[out=-90,in=65] (-8,-9) to[out=-85,in=125] (-10,-12.5) to[out=-55,in=95] (-11,-13.3); %to[out=-85,in=55] (-12.5,-15);

%Les lignes et les temps
\draw (0,0)--(8.8,0) node[midway, above]{$\sigma,\sigma'\in LO_0$}; %t=0

\draw[<->](8.5,0)--(8.5,-8) node[midway, right]{$(1-\varepsilon )t$} node[midway, left]{$t_0=$}; %ligne verticale

\draw[line width=8pt, dashed, gray, opacity=0.4] (-5,-8)--(7,-8) node[midway, above, opacity=1]{$\mathcal{Z}_0$ (cf Figure~\ref{fig:zeroslemma})}; %Z0

\draw[dashed](-6.1,-8)--(8.5,-8)node[right, scale=0.8]{$t_0$};;

\draw[<->](8.5,-8)--(8.5,-9) node[midway, right]{$\Delta$};
\draw[dotted](-7.2,-8.5)--(8.5,-8.5);
\draw[dashed](-7.8,-9)--(8.5,-9);

 \fill[white] (0,-8.5) circle(0.3);
 \draw[thick] (0,-8.5) circle(0.3) node[scale=0.7]{T};
 
\draw[<->](8.5,-9)--(8.5,-10) node[midway, right]{$\Delta$};
\draw[<->,dotted](8.5,-10)--(8.5,-11);% node[midway, right]{$\Delta$};
\draw[<->](8.5,-11)--(8.5,-12) node[midway, right]{$\Delta$};

\draw[dotted](-8,-9.5)--(8.5,-9.5);
  \fill[white] (0,-9.5) circle(0.3);
 \draw[thick] (0,-9.5) circle(0.3) node[scale=0.7]{T};
 
\draw[dashed](-8.3,-10)--(8.5,-10);
%BLANC entre -10 et -11

\draw[dashed](-9.5,-11)--(8.5,-11);
\draw[dotted](-9.8,-11.5)--(8.5,-11.5);
\fill[white] (0,-11.5) circle(0.3);
 \draw[thick] (0,-11.5) circle(0.3) node[scale=0.7]{S};

\draw[dashed](-10,-12)--(8.5,-12);
\draw[dotted](-9.8,-12.5)--(8.5,-12.5);
  \fill[white] (0,-12.5) circle(0.3);
 \draw[thick] (0,-12.5) circle(0.3) node[scale=0.7]{S};
%\draw[dotted](-9.9,-13)--(8.5,-13);
\draw[<->](8.5,-12)--(8.5,-13) node[midway, right]{$\Delta$};

\draw[dashed] (-10.7,-13)--(10,-13) node[right,scale=0.8]{$t_N$}; %t_N
\draw[dashed,line width=2pt] (-11,-13.3)--(11,-13.3) node[right,scale=0.9]{$t$}; %derniere ligne
 
%fronts
\draw[thick] (-11,-13.3) circle(0.2);% node[left]{$X(\sigma_t)$};
 \fill[white] (-11,-13.3) circle(0.2);
 \draw[thick] (0,0) circle(0.2);
  \fill[white] (0,0) circle(0.2);

\end{tikzpicture}

\par}
\caption{Coupling of the evolutions from distinct initial configurations: $N$ repetitions of the special coupling during a time lag $\Delta$. The labels T correspond to trials (Steps 1-2, detailed in Figure~\ref{fig:couplingzoom}), where the coupling attempts to match the two configurations. After the first success, the standard coupling maintains the matching up to time $t$; labels S refer to the first item in the coupling construction (see Figure~\ref{fig:deperdition}).} 
\label{fig:coupling}
\end{figure}
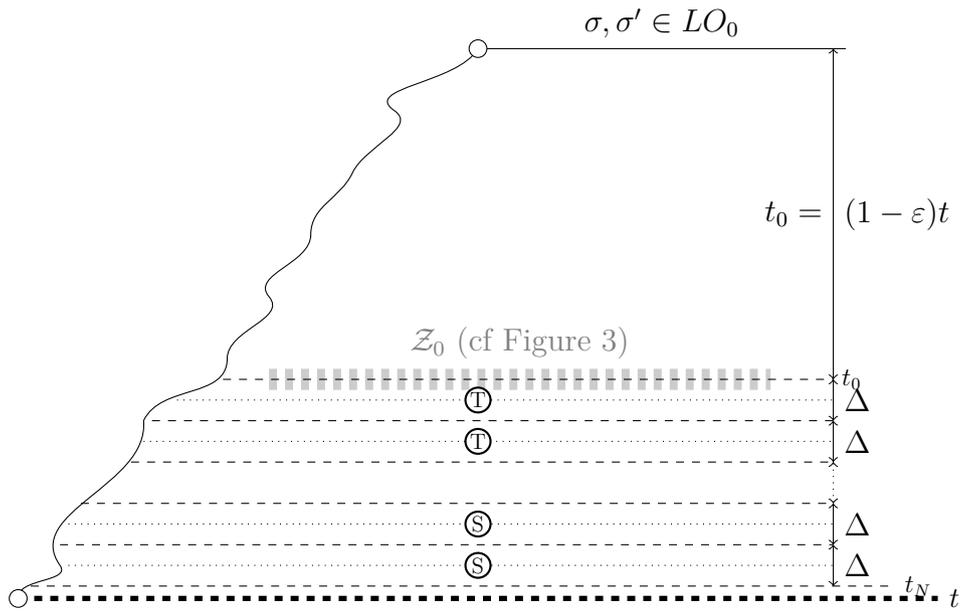

\begin{figure}[p]
{\centering

\begin{tikzpicture}[scale=0.7]
%LES LIGNES
\draw[-] (0,0)--(20,0);
\draw[-] (3,3)--(20,3);
\draw[-] (8,7)--(20,7);
%LES TEMPS
\draw (20.5,7) node[right]{$t_n=t_0+n\Delta$};
\draw (20.5,3) node[right]{$s_{n+1}=t_n+\Delta_1$};
\draw (20.5,0) node[right]{$t_{n+1}=s_{n+1}+\Delta_2$};
\draw[<->] (20.2,7)--(20.2,3) node[right, midway]{$\Delta_1= e^{(\log t)^{1/4}}\ll t^{\epsilon}$};
\draw[<->] (20.2,3)--(20.2,0) node[right, midway]{$\Delta_2= (\log t)^{3/4}$};
%\draw[<->] (20.4,7)--(20.4,0) node[right, midway]{$\Delta$};

%l'intervalle de zeros Zn (WCS)
\draw[line width=8pt, dashed, gray, opacity=0.6] (9,7)--(19.5,7) node[midway, above, opacity=1]{$\mathcal{Z}_n$} node[midway, below, opacity=1]{$\mathcal{H}(\sqrt{\Delta_1})$};
\draw[<->,gray] (8,7.5)--(9,7.5) node[gray,above,midway,scale=0.8]{$\underline{v} \Delta_1$};%$t_n^{\epsilon}$};   %D_1
\draw (19.5,7) node[above,scale=0.8]{$2\underline{v}t_n$};%${$v_{\min}t_n$};
\draw[dotted] (9,7)--(11,3);
\draw[dotted] (19.5,7)--(17.5,3)--(16,0);

%%%%%%%%%%%%   temps intermediaire

\fill[opacity=0.5,pattern=crosshatch] (3,2.8) rectangle (11,3.2);
\draw (5,3) node[below, opacity=1,scale=0.8]{indep coupling};
%\draw[line width=8pt, green, opacity=0.2] (3,3)--(11,3)  node[near start, below, opacity=1,scale=0.8]{indep coupling};

\fill[opacity=0.7,pattern=crosshatch] (17.5,2.8) rectangle (20,3.2);
%\draw[line width=8pt, green, opacity=0.2] (17.5,3)--(20,3);

%l'intervalle Lamnda_n 
%\draw[line width=8pt, red, opacity=0.2] (11,3)--(17.5,3) node[midway, above, opacity=1]{$\Lambda_n$} node[midway, below, opacity=1, scale=0.8]{maximal coupling};
\fill[gray!50,opacity=0.5,pattern=north east lines] (11,2.8) rectangle (17.5,3.2);
 \draw(14,3)  node[below, opacity=1, scale=0.8]{maximal coupling} node[above, opacity=1]{$\Lambda_n$}; 

%\draw (11,3) node[above,scale=0.8]{$D'_1=3\overline{v}\Delta_1$};

\draw[line width=8pt, dashed,  opacity=0.6] (11,3)--(12,3);
%\draw[<->, red] (12,3.5)--(14,3.5) node[above, near start, scale=0.6]{$v_{\max}\Delta_2$};

%l'intervalle  de zeros Z'n (SSC)
\draw[line width=8pt, dashed, gray, opacity=0.6] (5,3)--(11,3) node[midway, above, opacity=1]{$\mathcal{Z}'_n$} node[midway, below, opacity=1]{$\mathcal{H}(\sqrt{\Delta_2}/2)$};
\draw[gray,<->] (3.1,3.5)--(5,3.5) node[above,midway,gray,scale=0.8]{$ \sqrt{\Delta_2}/2$};%$\kappa\epsilon\log t$};  %D_2

\draw[<->, gray] (9,3.5)--(11,3.5) node[midway, above,scale=0.8]{$\overline{v}\Delta_1$}; 
\draw[<->, gray] (17.5,3.5)--(19.5,3.5) node[midway, above,scale=0.8]{$\overline{v}\Delta_1$};
\draw[<-, black] (17.5,2.8)--(18,2.2) node[black, below,scale=0.7]{$\geq 2\underline{v}s_{n+1}-(\overline{v}+\underline{v})\Delta_1$};
%\draw (17.5,3) node[below, scale=0.8]{$2\underline{v}s_{n+1}-2\overline{v}\Delta_1$};%$v_{\min}s_n$};

%%%%%%%%%%%%%    temps final

 \draw[line width=6pt, gray!90,opacity=0.5] (12,0)--(20,0);
 \draw (16,0) node[below,scale=0.8]{basic coupling};
 
 %\draw[line width=8pt, red, opacity=0.2] (3,0)--(12,0)  node[midway, below, opacity=1, scale=0.8]{maximal coupling}; %node[midway, above, opacity=1]{?}
 \fill[gray!50,opacity=0.5,pattern=north east lines] (3,-0.2) rectangle (12,0.2);
 
 \draw(7.5,0)  node[below, opacity=1, scale=0.8]{maximal coupling}; 
  
 %\draw[<->,red] (7,-1)--(12,-1) node[midway, below]{$e^{\Delta_2^\alpha}$};
  %\draw[<->,gray] (5,-1)--(7,-1) node[midway, below,scale=0.8]{$v_{\max}\Delta_2$};
%?\draw (7,0) node[above]{?};

  \draw[line width=6pt, gray!90,opacity=0.5] (0,0)--(3,0);
  \draw (1.5,0) node[below, scale=0.8]{basic coupling};
 
% \draw[red, dotted] (14,3)--(12,0); 
%  \draw[gray, dotted] (5,3)--(7,0); 
%ZERO COMMUN 
 \draw[line width=2pt] (12,3)--(12,0);
 \draw[thick] (12,3) circle(0.2);
 \fill[white] (12,3) circle(0.2);
  \draw[thick] (12,0) circle(0.2) node[below]{$x^*$};
 \fill[white] (12,0) circle(0.2);
 \draw[line width=2pt] (3,3)--(3,0);
 \draw[thick] (3,0) circle(0.2);
  \fill[white] (3,0) circle(0.2);

\draw[<->] (11,2.5)--(12,2.5) node[below, near start, scale=0.6]{$\le \frac{\sqrt{\Delta_2}}{2}$};

%%%%%%%% trajectoire du front

 \draw[-] (0,0) to[out=30,in=-145] (3,3); %to[out=-90,in=45] (-10,0);
 \draw[-] (3,3) to[out=90,in=-100] (8,7);% to[out=-60,in=90] (-7,8) to[out=-90,in=45] (-10,0);

%%%%%%%%   LES FRONTS
 \draw[thick] (0,0) circle(0.3);
 \fill[white] (0,0) circle(0.3);
 
 \draw[thick] (3,3) circle(0.3);
 \fill[white] (3,3) circle(0.3);
 
 \draw[thick] (8,7) circle(0.3);
 \fill[white] (8,7) circle(0.3);

\draw[<->] (0,-1)--(16,-1) node[midway, below]{$\geq I_{n+1}=2\underline{v}t_{n+1}-(\overline{v}+\underline{v}){\Delta}(n+1)$};

%%%%% LEGENDES

\fill[pattern=crosshatch] (0,-3) rectangle (1,-2.5);
\draw (1.1,-2.75) node[right]{independent coupling};

\fill[gray!90,opacity=0.5] (0,-4) rectangle (1,-3.5);
\draw (1.1,-3.75) node[right]{basic coupling};

\fill[gray!50,opacity=0.5,pattern=north east lines] (0,-5) rectangle (1,-4.5);
\draw (1.1,-4.75) node[right]{maximal coupling};

\draw[line width=8pt, dashed, gray, opacity=0.6] (0,-5.75)--(1,-5.75) node[opacity=1, black, right]{zeros property};

\end{tikzpicture}

\par}
\caption{Coupling of the evolutions from distinct initial configurations: Steps 1 and 2.}
\label{fig:couplingzoom}
\end{figure}
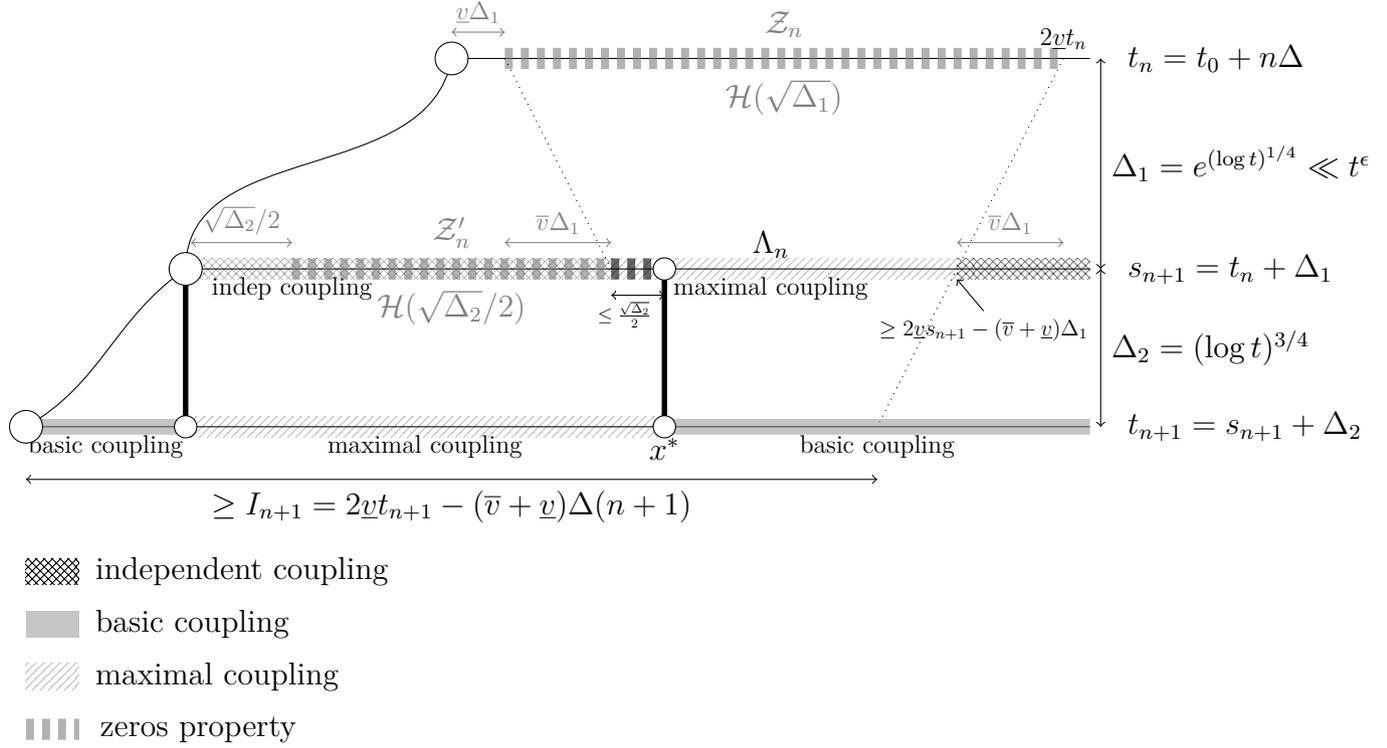

During the description of the precise procedure, we will use the basic coupling which consists in making the two configurations evolve according to the graphical representation using the same Poisson clocks and coin tosses (as we did in Section~\ref{graphical} between the FA-1f and contact processes). Note that whenever we use the basic coupling in our construction, we mean the basic coupling between two FA-1f processes with generator given by \eqref{genFA} (and not to processes ``seen from the front'').
We will also use a trickier coupling: the $\Lambda-$maximal coupling, denoted by $MC_{\Lambda}(\mu,\mu')$ where $(\mu,\mu')$ are two probability measures on $\Omega$ and $\Lambda$ is a finite box of $\Z$. It is defined as follows:
\begin{enumerate}
\item we sample $(\sigma,\sigma')_{|\Lambda\times\Lambda}$ according to the maximal coupling (which achieves the total variation distance see e.g.~\cite{levin}) of the marginals of $\mu$ and $\mu'$ on $\Omega_{\Lambda}$;
\item we sample $\sigma_{|\Z\setminus\Lambda}$ and $\sigma_{|\Z\setminus\Lambda}'$ independently according to their respective conditional distribution $\mu(\cdot|\sigma_{|\Lambda})$ and $\mu'(\cdot|\sigma'_{|\Lambda})$.
\end{enumerate}
We are now ready to recursively define a coupling $\mathcal{P}^{(n)}$ for $(\tilde{\mu}^{\sigma}_{t_n}, \tilde{\mu}^{\sigma'}_{t_n})$, with $n\in\{0,\ldots,N\}$, that is, a coupling $(\tilde{\sigma}_{t_n},\tilde{\sigma}'_{t_n})$ between the configurations \emph{seen from the fronts} at time $t_n$ (\textit{i.e.} $\tilde{\sigma}_{t_n}\sim\theta_{X(\sigma_{t_n})} \sigma_{t_n}$ and $\tilde{\sigma}'_{t_n}\sim\theta_{X(\sigma'_{t_n})} \sigma'_{t_n}$).

\begin{itemize}
\item \textbf{(Step 0)} %We choose product because we do not use any property of coupling for this step 
$\mathcal{P}^{(0)}$ is the trivial product coupling: to sample $(\tilde{\sigma}_{t_0},\tilde{\sigma}'_{t_0})$ we run the FA-1f dynamics starting from $(\sigma,\sigma')$ according to the product coupling and we take the configurations seen from the fronts at time $t_0$.
\item $\mathcal{P}^{(n)}\to\mathcal{P}^{(n+1)}$: we sample $(\tilde{\sigma}_{t_n},\tilde{\sigma}'_{t_n})$ according to $\mathcal{P}^{(n)}$:
\begin{enumerate}
\item %At the beginning of step 1+2 we check if we need to do it
If the configurations coincide on the interval $I_n=[1,d_n]$, where $d_n=2\underline{v}t_n-(\overline{v}+\underline{v})\Delta n$, then we let them evolve according to the basic coupling during a time lag $\Delta$: we obtain $(\tilde{\sigma}_{t_{n+1}},\tilde{\sigma}'_{t_{n+1}})$ by running the basic coupling started from $(\tilde{\sigma}_{t_n},\tilde{\sigma}'_{t_n})$ for a time $\Delta$, and then taking the configurations obtained as seen from the front. Using the basic coupling ensures that the equality of the configurations seen from the front is preserved w.h.p. on the interval $I_{n+1}=[0,d_{n+1}]$. By contrast, beyond distance  $d_n-(\overline{v}-\underline{v})\Delta$ from the front,  %,2\underline{v}t_{n+1}-2\overline{v}\Delta n]$ 
the information from beyond $I_n$ could have propagated, so we can not ensure equality. See Figure~\ref{fig:deperdition}.

\begin{figure}
\begin{center}
\includegraphics[scale = .5]{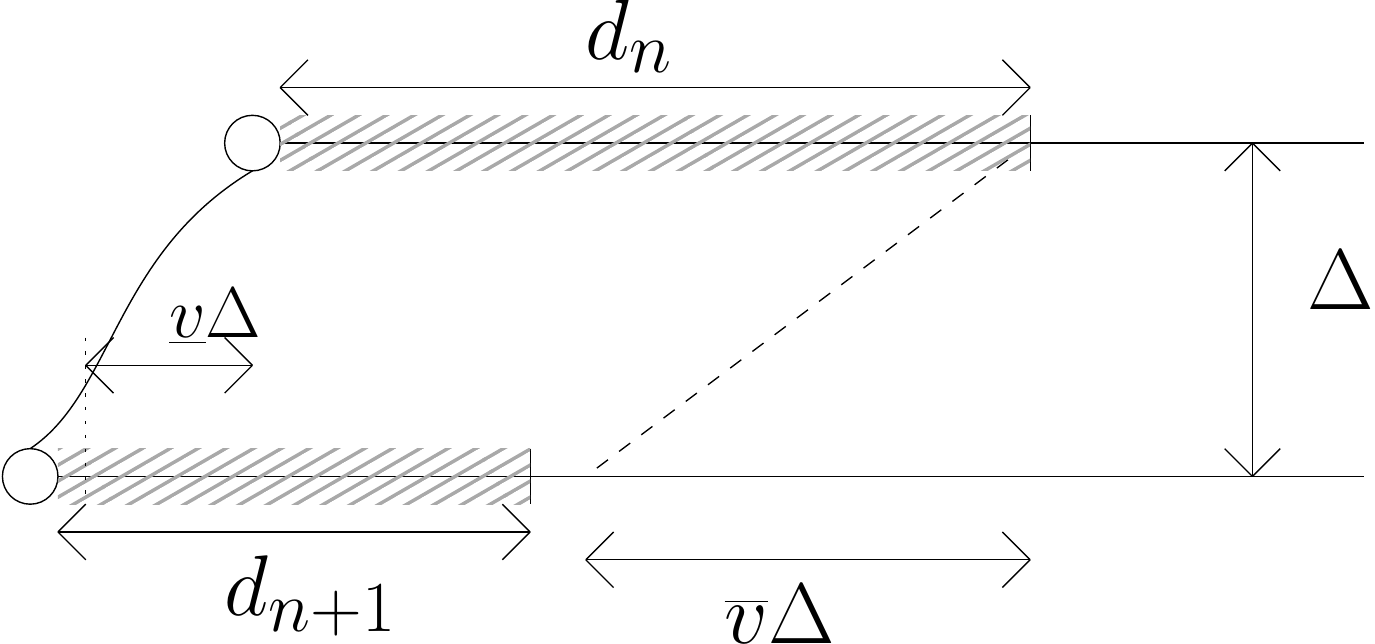}
\caption{With high probability, equality between two configurations in $LO_0$ in the dashed interval at initial time results in equality in the dashed interval after time $\Delta$. Indeed, discrepancies coming from the right travel at most at speed $\overline{v}$, and the front(s) move at least at speed $\underline{v}$.}
\label{fig:deperdition}
\end{center}
\end{figure}

\item If they do not coincide, we proceed in two steps. 
 \textbf{(Step 1)} We first choose $(\tilde{\sigma}_{s_{n+1}},\tilde{\sigma}'_{s_{n+1}})$ using the $\Lambda_n$-maximal coupling between the laws of the configurations \emph{seen from the front} after a time $\Delta_1$ starting from $\tilde{\sigma}_{t_n}$ and $\tilde{\sigma}'_{t_n}$, with $\Lambda_n:=[3\overline{v}\Delta_1,2\underline{v}s_{n+1}-(\overline{v}+\underline{v})\Delta_1]$. 
\begin{enumerate}
\item  If $\tilde{\sigma}_{s_{n+1}}$ and $\tilde{\sigma}'_{s_{n+1}}$ are not equal in the interval $\Lambda_n$, then we let them evolve for a time $\Delta_2$ via the basic coupling.
\item \textbf{(Step 2)} If instead they agree on $\Lambda_n$, then we search for the leftmost common zero of $\tilde{\sigma}_{s_{n+1}}$ and $\tilde{\sigma}'_{s_{n+1}}$ in $\Lambda_n$ and call $x^*$ its position. If there is no such zero, we define $x^*$ as the right boundary of $\Lambda_n$. Then we sample an independent Bernoulli random variable $\beta$ with $\P(\beta=1)=e^{-2\Delta_2}$. The event $\{\beta=1\}$ corresponds to the fact that the two Poisson clocks associated with $x^*$ and with the origin in the graphical construction do not ring during step 2. Using $(\tilde{\sigma}_{s_{n+1}},\tilde{\sigma}'_{s_{n+1}})$ as initial configurations, we sample two new configurations as follows.
 \begin{enumerate}
  \item If $\beta=1$ then we fix the configurations at $x^*$ as $\tilde{\sigma}_{s_{n+1}}(x^*),\tilde{\sigma}'_{s_{n+1}}(x^*)$. On $[1,x^*-1]$, we sample the two configurations using the maximal coupling for the FA-1f process run during time $\Delta_2$, with boundary conditions at $x^*$ being $\tilde{\sigma}_{s_{n+1}}(x^*)$ and $\tilde{\sigma}'_{s_{n+1}}(x^*)$ respectively, and $0$ at the origin. Finally, we sample the configurations on the right of $x^*$ and on the left of the origin using the basic coupling during time $\Delta_2$ with the same boundary conditions. Let $\xi_{n+1}$ be the (common) increment of the front during that time. See Figure~\ref{fig:couplingzoom}.
  %$\sigma(t_{n+1})_{x^*}=\sigma_{x^*}(s_{n+1})$ and 
%$\sigma'(t_{n+1})_{x^*}=\sigma'_{x^*}(s_{n+1})$. 
%The remaining part of the configurations at time $t_{n+1}$ is sampled using the basic coupling to the right of  $x^*$ and to the left of the origin; the maximal coupling for the FA process is used on the interval $[1,x^*-1]$ with boundary condition at $x^*$ equal to $\tilde{\sigma}_{s_{n+1}(x^*)$.
\item If $\beta=0$ then  we let evolve $(\tilde{\sigma}_{s_{n+1}},\tilde{\sigma}'_{s_{n+1}})$ for a time $\Delta_2$ via the basic coupling conditioned to have at least one ring either at $x^*$ or at the origin (or both).
  \end{enumerate}
Once procedure (a) or (b) is completed, we define $(\tilde{\sigma}_{t_{n+1}},\tilde{\sigma}'_{t_{n+1}})$ as the version seen from their fronts of the configurations we obtained.
\end{enumerate}
\end{enumerate}
\end{itemize}
The final coupling $\mathcal{P}^t_{\sigma,\sigma'}$ is obtained by sampling $(\tilde{\sigma}_{t_N},\tilde{\sigma}'_{t_N})$ according to $\mathcal{P}^{(N)}$, then applying the basic coupling during a time $t-t_N$ and then defining $(\tilde\sigma_t,\tilde\sigma'_t)$ as the configurations seen from the fronts at that time.

%%%%%%%%%%%%%%%%%%%%%%%%%%%%%%%%%%%%%%%%%%%%%%%%%%%%%%%%%%%%%%%%%%%%%%%%%%%%%%%%%%%%%%%%%%%%
 In the following, $\mathcal{P}$ denotes the joint distribution of all the variables used to define the coupling above. It will also be useful to introduce the following notation: for $A\subset LO_0\times LO_0$ measurable, $n=1,\ldots,N$, if $\mathcal{A}_n=\{(\tilde\sigma_{t_n},\tilde\sigma'_{t_n})\in A\}$, we write
\[
\mathcal{P}_{\mathcal{A}_n}(\cdot):=\sup_{(\eta,\eta')\in A}\mathcal{P}\left(\cdot\mid (\tilde\sigma_{t_n},\tilde\sigma'_{t_n})=(\eta,\eta')\right)
\]
and 
\[
\mathcal{P}_{\mathcal{A}_n,\beta=1}(\cdot):=\sup_{(\eta,\eta')\in A}\mathcal{P}\left(\cdot\mid (\tilde\sigma_{t_n},\tilde\sigma'_{t_n})=(\eta,\eta'),\beta=1\right).
\]
We define these quantities similarly if $\mathcal{A}_n=\{(\tilde\sigma_{s_n},\tilde\sigma'_{s_n})\in A\}$.

We can now study the probabilities of the events that we expect to see at the step $n$ of the previous procedure.

\begin{itemize}
\item At time $t_n$, we hope for enough zeros, that is, for the occurence of the event  
\begin{align*}
\mathcal{Z}_n&=\left\{\text{no 0-gap greater than }\sqrt{\Delta_1}\text{ on }[\underline{v}\Delta_1, 2\underline{v}t_n]\text{ behind the front at time }t_n\right\}\\
&:=\left\{\tilde{\sigma}_{t_n},\tilde{\sigma}'_{t_n}\in \mathcal{H}\left(\underline{v}\Delta_1,2\underline{v}t_n, \sqrt{\Delta_1}\right)\right\}.
\end{align*}
The distance $\sqrt{\Delta_1}$ is the maximal space that we allow between two zeros to relax to equilibrium at the next step; the distance $2\underline{v}t_n$ is the maximal distance where we can expect zeros without making any hypothesis on the initial configuration. By the first case of Lemma~\ref{zeros_lemma} we have 
\begin{align}
\label{firstzeros}
\mathcal{P}(\mathcal{Z}_n)%=\mathcal{P}^{(n)}(\mathcal{Z}_n)
\geq 1-2(2\underline{v}t_n)^2 e^{-c\sqrt{\Delta_1}}.
\end{align}
%%%%%%%%%%%%%%%%%%%%%%%%%%%%%%%%%%%%%%%%%%%
\item At time $s_{n+1}$, conditioning on the event $\mathcal{Z}_n$, we first can relax on the interval seen from the front $\Lambda_n=[3\overline{v}\Delta_1, 2\underline{v}s_{n+1}-(\overline{v}+\underline{v})\Delta_1]$ with error $O(e^{-c'\sqrt{\Delta_1}})$ in term of total variation by Theorem~\ref{thm:decorrelation}. %The distance $3\overline{v}\Delta_1$ is bigger that the sum of the maximal left deviation of the front, the distance without zeros at time $t_n$ and the distance which we loose where we try to relax on $\Lambda_n$. So, looking at distance $3\overline{v}\Delta_1$ we ensure that the configurations are close to equilibrium. With the same kind of computation, they are close to equilibrium on a distance at least $2\underline{v}s_{n+1}-2\overline{v}\Delta_1$ behind the front. 
Since the distributions of both configurations are close to $\mu$ on $\Lambda_n$, they are also close to each other, and denoting by $\mathcal Q_n =\{\tilde{\sigma}_{s_{n+1}}=\tilde{\sigma}_{s_{n+1}}'\text{ on }\Lambda_n\}$ and sampling the configurations according to the $\Lambda_n$-maximal coupling, we have 
\begin{align}
\label{firsteq}
\mathcal{P}_{\mathcal{Z}_n}(\mathcal{Q}_n^c)%\leq\sup_{\eta,\eta'\in\mathcal{H}\left(\underline{v}\Delta_1,2\underline{v}t_n, \sqrt{\Delta_1}\right)}\mathcal{P}\left(\mathcal Q_n^c\mid\tilde{\sigma}_{t_n}=\eta,\tilde{\sigma}'_{t_n}=\eta'\right)
% \le 2 \todo{...}
= O(e^{-c'\sqrt{\Delta_1}}).
\end{align}
Let %$x^*$ be the leftmost common zero of $\tilde{\sigma}(s_{n+1})$ and $\tilde{\sigma}'(s_{n+1})$ in $\Lambda_n$ (or the right boundary of $\Lambda_n$ if there is no such a zero and 
$\mathcal{B}_n$ the event corresponding to the fact that $x^*$ is at distance less than $\frac{\sqrt{\Delta_2}}{2}$ from the left boundary of $\Lambda_n$, that is
\[
\mathcal{B}_n:=\{x^*\leq 3\overline{v}\Delta_1+\frac{\sqrt{\Delta_2}}{2}\}.
\]
On $\mathcal{Q}_n$, $\mathcal{B}_n$ is implied by 
\[
\tilde{\mathcal{B}}_n:=\{(\tilde{\sigma}_{s_{n+1}})_{\mid [3\overline{v}\Delta_1,3\overline{v}\Delta_1+\frac{\sqrt{\Delta_2}}{2}]}\not\equiv 1\}.
\]
By Theorem~\ref{thm:decorrelation},
\begin{equation*}
\mathcal{P}_{\mathcal{Z}_n}(\tilde{\mathcal{B}}_n^c )\leq p^{\frac{\sqrt{\Delta_2}}{2}}+O(e^{-c'\sqrt{\Delta_1}}),
\end{equation*}
and therefore %Conditionning on to equilibrium we have 
\begin{align}
\label{rightmostzero}
\mathcal{P}_{\mathcal{Z}_n}\left(\mathcal{B}_n^{c}\cap\mathcal{Q}_n\right)\leq p^{\frac{\sqrt{\Delta_2}}{2}}+O(e^{-c'\sqrt{\Delta_1}}),
\end{align}
%with definitions of $\mathcal{P}_{\mathcal{Z}_n}(\tilde{\mathcal{B}}_n^c)$ and $\mathcal{P}_{\mathcal{Z}_n\cap\mathcal{Q}_n}\left(\mathcal{B}_n^{c}\right)$ similar to that in \eqref{firsteq}.
Then, we also use the zeros at time $t_n$ to generate zeros between the front and $\Lambda_n$ at time $s_{n+1}$. Actually the event $\mathcal{Z}_n$ implies that $\tilde{\sigma}_{t_n},\tilde{\sigma}'_{t_n}\in \mathcal{H}(0,3\overline{v}\Delta_1,2\underline{v}\Delta_1)$ because $\underline{v}\Delta_1+\sqrt{\Delta_1} \leq 2\underline{v}\Delta_1$. We consider the event 
\begin{align*}
\mathcal{Z}_n'&=\left\{\text{no $0$-gap greater than }\frac{\sqrt{\Delta_2}}{2}\text{ on }\left[\frac{\sqrt{\Delta_2}}{2}, 3\overline{v}\Delta_1\right]\text{ behind the front at time }s_{n+1}\right\}\\
&:=\left\{\tilde{\sigma}_{s_{n+1}},\tilde{\sigma}'_{s_{n+1}}\in \mathcal{H}\left(\frac{\sqrt{\Delta_2}}{2}, 3\overline{v}\Delta_1, \frac{\sqrt{\Delta_2}}{2}\right)\right\}.
\end{align*}
Applying the second case of Lemma~\ref{zeros_lemma} we have that:
\begin{align}
\label{secondzeros}
\mathcal{P}_{\mathcal{Z}_n}({\mathcal{Z}'_n}^c )\leq \frac{4\Delta_1^2}{\sqrt{\Delta_2}} e^{-c\frac{\sqrt{\Delta_2}}{2}}+6\overline{v}\Delta_1^2 e^{-c(\Delta_1\wedge \frac{\sqrt{\Delta_2}}{2})} =O(e^{-c'\sqrt{\Delta_2}}).
\end{align}
%%%%%%%%%%%%%%%%%%%%%%%%%%%%%%%%%%%%
\item  At time $t_{n+1}$ conditioning on the events $\mathcal{Z}_n'$, $\mathcal{Q}_n$, $\mathcal{B}_n$ and on the fixation of the zeros on $X(s_{n+1})$ and $x^*$ (i.e.\ $\beta=1$), we can relax to equilibrium on $[X(s_{n+1})+1, X(s_{n+1})+x^*-1]$, that is, on the interval $[1-\xi_{n+1}, x^*-1-\xi_{n+1}]$ seen from the front. So %, introducing the Bernoulli random variable $\beta$ corresponding to the fact that there is no clock rings at the front and at $x^*$ between time $s_{n+1}$, time $t_{n+1}$ and 
we have for all $y\in\Lambda_n$ at distance less than $\frac{\sqrt{\Delta_2}}{2}$ from the left boundary of $\Lambda_n$: 
\begin{align}
\label{secondeq}
\mathcal P_{\mathcal{Z}'_n,\mathcal{Q}_n,\mathcal{B}_n, x^*=y,\beta=1}\left(\tilde{\sigma}_{t_{n+1}}\neq\tilde{\sigma}_{t_{n+1}}'\text{ on }[1-\xi_{n+1}, x^*-1-\xi_{n+1}]\mid \right)
= O(e^{-c'\sqrt{\Delta_2}}),
\end{align}
applying~\eqref{eq:relaxation_finitevolume}. Furthermore, if the configurations coincide on $\Lambda_n$ at time $s_{n+1}$, then we can control the probability that they coincide on $\Lambda_n-\xi_{n+1}$ minus a sub interval of size $\overline{v}\Delta_2$ at time $t_{n+1}$ by the probability of the event $F(d_n,d_n-\overline{v}\Delta_2,\Delta_2)$ corresponding to the propagation of the information located beyond $\Lambda_n$. Indeed, by Lemma~\ref{finitespeedpropagation} and Corollary~\ref{atleastlinearFA}, we have (recall Figure~\ref{fig:deperdition}),
\begin{multline}
\label{degradation}
\mathcal P_{{\mathcal{Q}_n},{\beta=1}}\left(\tilde{\sigma}_{t_{n+1}}\neq\tilde{\sigma}_{t_{n+1}}'\text{ on }[ x^*+1-\xi_{n+1}, d_{n+1}]\right)
\leq \P\left(F(d_n,d_n-\overline{v}\Delta_2,\Delta_2)\right)\\
+\mathcal{P}(\xi_{n+1}\geq-\underline{v}\Delta_2)=O(e^{-c'{\Delta_2}}),
\end{multline}
where the first inequality comes from the fact that, if $\xi_n\leq-\underline{v}\Delta_2$, $d_n-\overline{v}\Delta_2-\xi_{n+1}\geq d_{n+1}$.
To conclude, notice that by construction
\begin{equation}\label{eq:trivial}
\mathcal P\left(\tilde{\sigma}_{t_{n+1}}\neq\tilde{\sigma}_{t_{n+1}}'\text{ on }[1, -\xi_{n+1}]\mid{\beta=1}\right)=0.
 \end{equation}
 Combining \eqref{secondeq}, \eqref{degradation} and \eqref{eq:trivial}, we get
 \begin{equation}\label{nosuccess}
\mathcal P_{\mathcal{Z}'_n,\mathcal{Q}_n,\mathcal{B}_n,\beta=1}\left(\tilde{\sigma}_{t_{n+1}}\neq\tilde{\sigma}_{t_{n+1}}'\text{ on }[1, d_{n+1}]\right)=O(e^{-c'\sqrt{\Delta_2}}).
\end{equation}
\end{itemize}

%%%%%%%%%%%%%%%%%%%%%%%%%%%%%%%%%%%

Let $\mathcal{M}_n$ be the event of matching in $I_n=[1,d_n]$ at time $t_n$:
\[
\mathcal M_n^c=\left\{\tilde{\sigma}_{t_n}\neq \tilde{\sigma}'_{t_n} \text{ on }I_n\right\},
\]
and denote by $p_n$ the probability that the coupling $\mathcal{P}^{(n)}$ is not successful, that is, $p_n=\mathcal{P}(\mathcal M_n^c)$. With the previous estimates in mind, we prove the following recursive relation: 
\begin{align}
\label{eq_rec}
p_{n+1}\leq Ce^{-c\Delta_1}+p_n(1- \frac{1}{2}e^{-2\Delta_2}).
\end{align}
Indeed, using the previously introduced event $\mathcal{Z}_n$, we have that
\begin{align}
\mathcal P\left(\mathcal{M}_{n+1}^c \right)
&\leq \mathcal P\left(\mathcal{M}_{n+1}^c \cap \mathcal{M}_{n}^c  \cap \mathcal{Z}_{n} \right)+\mathcal P\left(\mathcal{Z}_{n}^c\right) +\mathcal{P}(\mathcal{M}_{n+1}^c \cap \mathcal{M}_{n})\nonumber
\\
&\leq  \mathcal P\left(\mathcal{M}_{n+1}^c \,\big|\, \mathcal{M}_{n}^c  \cap \mathcal{Z}_{n} \right)\mathcal{P}(\mathcal{M}_{n}^c ) + \mathcal P\left(\mathcal{Z}_{n}^c\right) \nonumber
\\
&\quad+\P\left(F(d_n,d_n-\overline{v}\Delta,\Delta)\right)+\sup_{\omega\in LO_0}\P_{\omega}(X(\omega_\Delta)>-\underline{v}\Delta).\label{eqrecstep1}
\end{align}
The last two terms come from a reasoning similar to what we did in \eqref{degradation}. %correspond to the fact that some discrepancies have moved from outside of $I_n$ at time $t_n$ to inside of $I_{n+1}$ (seen from the front) at time $t_{n+1}$, destructing the matching. %Since $|d_{n+1}-d_n|\ge \overline{v}(t_{n+1}-t_n)$, b
By Lemma~\ref{finitespeedpropagation} and Corollary~\ref{atleastlinearFA}, their contribution is %we have $\mathcal{P}\left(F(d_n,d_{n+1},\Delta)\right)=
$O(e^{-c\Delta})$. The quantity $\mathcal P\left(\mathcal{Z}_{n}^c\right)$ has been controlled in~\eqref{firstzeros} so we now have to study 
 $\mathcal P\left(\mathcal{M}_{n+1}^c \,\big|\, \mathcal{M}_{n}^c  \cap \mathcal{Z}_{n} \right)$. %\le 
Now we condition by the favorable events at time $s_{n+1}$:
%\begin{multline*}
%\mathcal P\left(\mathcal{M}_{n+1}^c\mid \mathcal Z_n\cap\mathcal M_n^c \right)\leq \mathcal P_{\mathcal Z_n\cap\mathcal M_n^c}\left(\mathcal{M}_{n+1}^c \,\big|\, \mathcal{Z}'_{n}\cap\mathcal{Q}_n\cap\mathcal{B}_n \right)\\
%+\mathcal P_{\mathcal Z_n}\left(\mathcal{Z}_{n}'^c\right) +\mathcal P_{\mathcal Z_n}\left(\mathcal{Q}_{n}^c\right)+\mathcal P_{\mathcal Z_n\cap\mathcal Q_n}\left(\mathcal{B}_{n}^c \right).
%\end{multline*}
\begin{multline*}
\mathcal P\left(\mathcal{M}_{n+1}^c\mid \mathcal Z_n\cap\mathcal M_n^c \right)
\leq \mathcal P_{\mathcal Z_n\cap\mathcal M_n^c}\left(\mathcal{M}_{n+1}^c \,\big|\, \mathcal{Z}'_{n}\cap\mathcal{Q}_n\cap\mathcal{B}_n \right)
\\
+\mathcal P_{\mathcal{Z}_n}\left(\mathcal{Z}_{n}'^c \right) +\mathcal P_{\mathcal{Z}_n}\left(\mathcal{Q}_{n}^c \right)+\mathcal P_{\mathcal{Z}_n}\left(\mathcal{B}_{n}^c \cap\mathcal Q_n\right).
\end{multline*}
The last terms have been controlled by~\eqref{firsteq},~\eqref{rightmostzero} and~\eqref{secondzeros}.

Finally, {using~\eqref{nosuccess}}
\begin{align*}
\mathcal P&\left(\mathcal{M}_{n+1}^c \,\big|\, \mathcal{Z}'_{n}\cap\mathcal{Q}_n\cap\mathcal{B}_n \cap {\mathcal Z_n\cap\mathcal M_n^c}\right)\\
&\leq 1- \P(\beta=1)\left(1-\mathcal P\left(\mathcal M^c_{n+1} \,\big|\,\{\beta=1\}\cap\mathcal{Z}'_{n}\cap\mathcal{Q}_n
\cap\mathcal{B}_n\cap {\mathcal Z_n\cap\mathcal M_n^c}\right)\right)\\
&\leq 1-e^{-2\Delta_2}\frac{1}{2},
\end{align*}
for $t$ large enough. %and~\eqref{degradation} to control the term $\mathcal P_{\mathcal{M}_{n}^c  \cap \mathcal{Z}_{n}}(\mathcal M^c_{n+1}|\beta=1,\mathcal{Z}'_{n},\mathcal{B}_{n})$.
Combining with~\eqref{eqrecstep1} we obtain the desired recursivity~\eqref{eq_rec}. 
%\[p_{n+1}\leq Ce^{-c\Delta_1}+p_n(1- \frac{1}{2}e^{-2\Delta_2}).\]
So, we get 
\begin{align*}
p_N&\leq \left(1- \frac{1}{2}e^{-2\Delta_2}\right)^{N}+ Ce^{-c\Delta_1}e^{\Delta_2} %\\
%&=O(e^{-Ne^{-2\Delta_2}})+O(e^{-c'\Delta_1})\\
%&=(e^{-c'\Delta_1})=O(e^{-ce^{(\log t)^{1/4}}})
\end{align*}
%%%%%%%%%%%%%%%%%%%%%%%%%%%%%%%
and we conclude that $p_N=O(e^{-ce^{(\log t)^{1/4}}})$, replacing $\Delta_1,\Delta_2,N$ by their chosen values.

To conclude the proof of~\eqref{thmconfig}, we now compute the distance on which the matching occurs w.h.p. At time $t_N$ the coupling was successful on $I_N$ with probability $1-p_N$. Note that $d_N\geq (2\underline{v}-\epsilon(\overline{v}+\underline{v}))t-2\underline{v}\Delta$. Let $\epsilon=\frac{\underline{v}}{2(\overline{v}+\underline{v})}$ and $d^*=\underline{v}$. Then $d^*t\leq d_N$ and if we let the configurations evolve according to the basic coupling between time $t_N$ and time $t$ we have 
\begin{multline*}
\mathcal{P}\left(\tilde{\sigma}_{t}\neq \tilde{\sigma}'_{t} \text{ on }[0,d^*t] \right)
\leq p_N + \P\left(F(d_N,{d^*t},t-t_N)\right)+\sup_{\omega\in LO_0}\P\left(X(\omega_{t-t_N})\geq 0\right).
\end{multline*}
Since $d_N-d^*t\geq \underline{v}t/2-2\underline{v}\Delta\geq\overline{v}\Delta\geq \overline{v}(t-t_N)$, we have that $\P(F(d_N,d^*t,t-t_N))\leq e^{- {c}t}$ and we obtained the announced convergence.

The fact that the set of probability measures on $\Omega$ is compact gives us the existence of invariant measures. The uniqueness and convergence starting from any $\pi$ measure on $LO_0$ comes from~\eqref{thmconfig} and we conclude the proof of~Theorem~\ref{thm:coupling}.

\end{proof}

\section{Limit Theorems}

\subsection{Moments and covariances}\label{section:hypotheses}

Thanks to Theorem~\ref{thm:coupling} we can study the increments of the front. We consider $\sigma\in LO_0$. For $n\in\N$, we introduce the following increments
\begin{align*}
\xi_n=X(\sigma_{n})-X(\sigma_{{n-1}}),
\end{align*}
so that, for $t>0$, \[X(\sigma_t)=\sum_{n=1}^{\lfloor t\rfloor} \xi_{n}+ X(\sigma_t)-X(\sigma_{\lfloor t\rfloor}).\]
In order to prove a law of large numbers and a central limit theorem, we want to control the moments (Lemma~\ref{moment2xi}) and the covariances (Lemma~\ref{covxinu}) of these front increments. Lemma~\ref{momentxinu} proves the convergence of $\E_{\sigma}[\xi_n]$ to $\E_{\nu}[\xi_1]$, which leads to the fact that the covariance between $\xi_j$ and $\xi_n$ decays fast enough.

\begin{lemma}\label{moment2xi} For $f:\N\to\R$ such that $e^{-|x|}f(x)^2\in L^1$ we have 
\begin{align}\label{eq:moment2xi}
\sup_{\sigma\in LO_0}\E_{\sigma}\left[f(\xi_1)^2\right]=c(f)<\infty.
\end{align}
\end{lemma}
\begin{remark} Consequently, we have the same result for $\E_{\sigma}[f(\xi_n)^2]$, for the expectation under $\nu$ and for the covariances between increments.\end{remark}

\begin{proof}The result follows from the finite speed of propagation. Indeed, by Lemma~\ref{finitespeedpropagation} we have that, for $\sigma\in LO_0$ and $|x|\geq \overline{v}$
\[
\P_{\sigma}(\xi_1=x)\leq \P(F(0,x,1))\leq e^{-|x|}.
\]
Then, for every $\sigma\in LO_0$, we have
\begin{align*}
\E_{\sigma}\left[f(\xi_1)^2\right]&\leq \max_{|x|\leq \overline{v}}f(x)^2+\sum_{|x|>\overline{v}} f(x)^2\P_{\sigma}(\xi_1=x)\\
&\leq c(f)<\infty.
\end{align*}
The right bound does not depend on $\sigma$, so we obtain the announced result.
\end{proof}

\begin{lemma}\label{momentxinu} There exists $\gamma>0$ such that for $f:\N\to\R$ with $e^{-|x|}f(x)^2\in L^1$,
%l'inegalite: \[|\E_{\sigma}(f(\xi_n))-\E_{\nu}(\xi_1)|\leq  \sqrt{c(f)}\left[ 4e^{-\frac{d^*t_{n-1}}{2}}+O(err(t_{n-1}))\right] \leq  \sqrt{c(f)}\left[ 4e^{-\frac{d^*t_{n-1}}{2}}+e^{-ce^{(\log t_{n-1})^{1/4}}}\right]\]
\[
\sup_{\sigma\in LO_0}\big|\E_{\sigma}\left[f(\xi_n)\right]-\E_{\nu}\left[f(\xi_1)\right]\big|
\leq C(f)e^{-\gamma e^{(\log n)^{1/4}}}.
\]
\end{lemma}

\begin{proof} 
%$c'(f)= \sup_{\sigma}\sqrt{\E_{\sigma}(f(\xi_1)^2)}$
We use the Markov property at time $n-1$ to write
\[
\E_{\sigma}\left[f(\xi_n)\right]=\int d\mu^{\sigma}_{n-1}(\sigma')\E_{\sigma'}\left[f(\xi_1)\right],
\]
where $\mu^\sigma_t$ is the distribution of the configuration at time $t$ starting from $\sigma$.
Let $\Phi_t(\sigma')$ be the configuration equal to $\theta_{X(\sigma')}\sigma'$ on $[0,d^*t]$ and equal to $1$ elsewhere (we recall that $d^*$ is a quantity defined in Theorem~\ref{thm:coupling}). Under the basic coupling (denoted by $\mathcal{E}$ in the following computations), the front of the configuration after a time $\delta$ starting with $\sigma'$ is different from the one starting with $\Phi_t(\sigma')$ only if the event $\tilde{F}(0,d^*t,\delta)$ occurs. So, denoting by $X$ and $\tilde{X}$ the first increments starting respectively from configurations $\sigma'$ and $\Phi_{n-1}(\sigma')$, we have that

\begin{align*}
\E_{\sigma'}\left[f(\xi_1)\right]-\E_{\Phi_{n-1}(\sigma')}\left[f(\xi_1)\right] &=\mathcal{E}\left[\left(f(X)-f(\tilde{X})\right)\1_{X\neq \tilde{X}}\right]
\\
&\leq \sqrt{\mathcal{E}\left[\left(f(X)-f(\tilde{X})\right)^2\right]} \sqrt{\mathcal{E}\left[\1_{X\neq \tilde{X}}\right]}
\\
%&\leq 2 \sqrt{\E_{\sigma'}\left[f(\xi_1)^2\right]} \sqrt{\P_{\sigma'}\left(\xi_1(\sigma')\neq \xi_1(\Phi_{t_{n-1}}(\sigma'))\right)}
%\\
&\leq 2\sqrt{\sup_{\sigma\in LO_0}\E_{\sigma}\left[f(\xi_1)^2\right]} \sqrt{\P(\tilde{F}(0,d^*(n-1),\delta))}.
\end{align*}
Therefore, if $n-1\geq\frac{\overline{v}}{d^*}$, %then $d^*t_{n-1}\geq \overline{v}\delta$ and
\[
\sup_{\sigma\in LO_0} \int d\mu^{\sigma}_{n-1}(\sigma')\left(\E_{\sigma'}\left[f(\xi_1)\right]-\E_{\Phi_{n-1}(\sigma')}\left[f(\xi_1)\right]\right)\leq 2e^{-\frac{d^*(n-1)}{2}}\sqrt{\sup_{\sigma\in LO_0} \E_{\sigma}\left[f(\xi_1)^2\right]}.
\]
%In the same way, we introduce the measure $\Phi_t(\nu)$ equal to $\nu$ on $[0,d^*t]$ and equal to $\delta_1$ elsewhere. 
Similarly (taking the expectation w.r.t.\ $\nu$), we have that 
\[
\E_{\nu}\left[f(\xi_1)\right]-\int d\nu(\sigma')\E_{\Phi_{n-1}(\sigma')}\left[f(\xi_1)\right]
\leq 2e^{-\frac{d^*(n-1)}{2}}\sqrt{\sup_{\sigma\in LO_0} \E_{\sigma}\left[f(\xi_1)^2\right]}.
\]
Besides, we can write that
\begin{multline*}
\big|\int d\mu_{n-1}^{\sigma}(\sigma')\E_{\Phi_{n-1}(\sigma')}\left[f(\xi_1)\right]-\int d\nu(\sigma')\E_{\Phi_{n-1}(\sigma')}\left[f(\xi_1)\right]\big|
\\
\leq 2\sqrt{\sup_{\sigma\in LO_0}\E_{\sigma}\left[f(\xi_1)^2\right]}\sup_{\sigma\in LO_0}\left\|\mu^{\sigma}_{n-1}-\nu\right\|^{1/2}_{\left[0,d^*(n-1)\right]}.
\end{multline*}
Finally, putting everything together
\begin{align*}
\sup_{\sigma\in LO_0}\big|\E_{\sigma}\left[f(\xi_n)\right]-\E_{\nu}\left[f(\xi_1)\right]\big|
\leq  2\sqrt{c(f)}\left( 4e^{-\frac{d^*(n-1)}{2}}+\sup_{\sigma\in LO_0} \left\|\mu^{\sigma}_{n-1}-\nu\right\|^{1/2}_{\left[0,d^*(n-1)\right]}\right).
\end{align*}
Applying Theorem~\ref{thm:coupling} we obtain the desired control.
\end{proof}

\begin{lemma}\label{covxinu} There exists $\gamma>0$ such that, for $f:\N\to\R$, $e^{-|x|}f(x)^2\in L^1$ and $j<n$ two positive integers,
\begin{enumerate}
\item $\displaystyle \sup_{\sigma\in LO_0}\big|\Cov_{\sigma}\left[f(\xi_j),f(\xi_n)\right]\big|\leq C(f)e^{-\gamma e^{(\log (n-j))^{1/4}}}$,

and the same holds for the covariance under $\nu$;
\item  for $j\geq \frac{\overline{v}}{d^*}(n-j)$, 
\[
\displaystyle \sup_{\sigma\in LO_0}\big|\Cov_{\sigma}\left[f(\xi_j),f(\xi_n)\right]-\Cov_{\nu}\left[f(\xi_1),f(\xi_{n-j+1})\right]\big|\leq C(f)e^{-\gamma e^{(\log j)^{1/4}}}.
\]
\end{enumerate}
\end{lemma}

\begin{proof}
Let $j<n$ and $\sigma\in LO_0$. 
\begin{enumerate}
\item The idea is the following one: if the difference between $j$ and $n$ is large enough then $\xi_j$ and $\xi_n$ are almost independent. Let $(\mathcal{F}_{j})_{j\in\N}$ be the filtration associated with the random variables $(\xi_j)_{j\in\N}$. Using the Markov property and Cauchy-Schwarz inequality we have that 
\begin{align*}
\big|\Cov_{\sigma}\left[f(\xi_j),f(\xi_n)\right]\big|&=\big|\Cov_{\sigma}\left[f(\xi_j),\E_{\sigma}\left[f(\xi_n)|\mathcal{F}_{j}\right]\right]\big|
\\
&=\big|\Cov_{\sigma}\left[f(\xi_j),\E_{\sigma_{j}}\left[f(\xi_{n-j})\right]-\E_{\nu}\left[\xi_1\right] \right]\big|
\\
&\leq 2\sqrt{\E_{\sigma}\left[f(\xi_j)^2\right]}\sqrt{\E_{\sigma}\left[\left(
\E_{\sigma_{j}}\left[f(\xi_{n-j})\right]-\E_{\nu}\left[\xi_1\right] \right)^2\right]}.
\end{align*}
Then, we apply Lemma~\ref{momentxinu} and we obtain the first point: 
\[\sup_{\sigma\in LO_0}\big|\Cov_{\sigma}\left[f(\xi_j),f(\xi_n)\right]\big| = O\left(e^{-\gamma e^{(\log (n-j))^{1/4}}}\right),
\]
which is relevant if $n-j$ is large. Combining this with Lemma~\ref{momentxinu} yields the result for the covariance under $\nu$.
\item If $n-j$ is not large but $j$ (and $n$) are, we are looking at the process after large times so the configurations seen from the front are close to configurations sampled by $\nu$. We can use the same trick as in the proof of the previous lemma, replacing the expectations by the covariances to obtain the second point. Indeed, we have 
\begin{multline*}
\Cov_{\sigma}\left[f(\xi_j),f(\xi_n)\right]-\Cov_{\nu}\left[f(\xi_1),f(\xi_{n-j+1})\right]
\\
= \E_{\sigma}\left[f(\xi_j)f(\xi_n)\right]-\E_{\nu}\left[f(\xi_1)f(\xi_{n-j+1})\right]
\\
-\big( \E_{\sigma}\left[f(\xi_j)\right]\E_{\sigma}\left[f(\xi_n)\right]
-\E_{\nu}\left[f(\xi_1)\right]\E_{\nu}\left[f(\xi_{n-j+1})\right]\big).
\end{multline*}
Using the Markov property at time $t_{j-1}$, finite speed propagation under the hypothesis $d^*(j-1)\geq \overline{v}(n-j+1)$ and Theorem~\ref{thm:coupling}, we can show that 
%$d^* t_{j-1}\geq \overline{v}t_{n-j+1}$
\[
|\E_{\sigma}\left[f(\xi_j)f(\xi_n)\right]-\E_{\nu}\left[f(\xi_1)f(\xi_{n-j+1})\right]|=O\left(e^{-\gamma e^{(\log j)^{1/4}}}\right).
\]
We conclude using Lemma~\ref{momentxinu} which says that
\begin{align*}
|\E_{\sigma}\left[f(\xi_j)\right]-\E_{\nu}\left[f(\xi_1)\right]|&=O(e^{-\gamma e^{(\log j)^{1/4}}}) 
\\
\text{ and }|\E_{\sigma}\left[f(\xi_n)\right]-\E_{\nu}\left[f(\xi_{n-j+1})\right]|&=O(e^{-\gamma e^{(\log n)^{1/4}}}).\qedhere
\end{align*}
\end{enumerate}
\end{proof}
The second point in the previous lemma can be generalized in the following way.

\begin{lemma}\label{momentgal}
For any $k,n\in\N$ such that $d^*(k-1)\geq \overline{v}n$ and any bounded function $F:\R^n\to \R$
\[
\sup_{\sigma\in LO_0}\big|\E_{\sigma}\left[F(\xi_k,\ldots, \xi_{k+n-1})\right]-\E_{\nu}\left[F(\xi_1,\ldots, \xi_{n})\right]\big|=O\left(\|F\|_{\infty}e^{-\gamma e^{(\log k)^{1/4}}}\right).
\]
\end{lemma}
\begin{proof} We apply again the method we used for Lemma~\ref{momentxinu}. Applying Markov property at time $t_{k-1}$ we have $\E_{\sigma}\left[F(\xi_k,\ldots,\xi_{k+n-1})\right]=\int d\mu^{\sigma}_{t_{k-1}}(\sigma')\E_{\sigma'}\left[F(\xi_1,\ldots,\xi_n)\right]$. We use the same notation $\Phi_t(\sigma)$ for the configuration equal to $\theta_{X(\sigma)}\sigma$ on $[0,d^*t]$ and equal to $1$ elsewhere; then %$X$ and $\tilde{X}$ are the random variable $(\xi_1,\ldots,\xi_{n})$ when the initial configurations are respectively $\sigma$ and $\Phi_t(\sigma)$. So 
\begin{align*}
|\E_{\sigma'}\left[F(\xi_1,\ldots,\xi_n)\right]-&\E_{\Phi_{k-1}(\sigma')}\left[F(\xi_1,\ldots,\xi_n)\right]| \\
&\leq 2\sqrt{\sup_{\sigma\in LO_0}\E_{\sigma}\left[F(\xi_1,\ldots,\xi_n)^2\right]} \sqrt{\P(F(0,d^*(k-1),n))}
\\
&\leq 2\|F\|_{\infty}e^{-\frac{d^*(k-1)}{2}}\text{ if }d^*(k-1)\geq \overline{v}n.
\end{align*}
Therefore,
\begin{multline*}
\big|\int d\mu_{k-1}^{\sigma}(\sigma')\E_{\Phi_{k-1}(\sigma')}\left[F(\xi_1,\ldots,\xi_n)\right]-\int d\nu(\sigma')\E_{\Phi_{k-1}(\sigma')}\left[F(\xi_1,\ldots,\xi_n)\right]\big|
\\
\leq 2\|F\|_{\infty}\sup_{\sigma\in LO_0}\left\|\mu^{\sigma}_{k-1}-\nu\right\|^{1/2}_{\left[0,d^*(k-1)\right]}.
\end{multline*}
Finally, putting everything together and applying Theorem~\ref{thm:coupling}, we obtain the announced result.\end{proof}

\subsection{Law of large numbers and central limit theorem}

\begin{proof}[Proof of Theorem~\ref{thm:limit}] We first proove~\eqref{eq:lgn}. We take inspiration from the proof of the law of large numbers of~\cite{GK}. %\[X(\sigma_t)-vt=S_n + X(\sigma_t)-X(\sigma_{t_{N}})\]
Let $S_n=\sum_{i=1}^{n} \tilde{\xi}_{i}$ where $\tilde{\xi}_{i}={\xi}_{i}-\E_{\sigma}\left[{\xi}_{i}\right]$ and $\Var_{\sigma} [\tilde{\xi}_{i}]\leq \E_{\sigma}\left[{\xi}_{i}^2\right]\leq c$ by Lemma~\ref{moment2xi}.
First we compute the variance of $\frac{S_n}{n}$:
\begin{align*}
\Var_{\sigma}\left[\frac{S_n}{n}\right]&=\frac{1}{n^2}\sum_{i=1}^n \Var_{\sigma} \left[\tilde{\xi}_{i}\right]
+\frac{2}{n^2}\sum_{j<k}^n \Cov _{\sigma}\left[\tilde{\xi}_{j},\tilde{\xi}_{k}\right].
\end{align*}
The first sum is dominated by $cn$. We control the second one thanks to Lemma~\ref{covxinu}. %Let $n_0$ be the quantity from Lemma~\ref{covxinu}:
\begin{align*}
\sum_{k=1}^n \sum_{j=1}^{k-1} \Cov _{\sigma}\left[\tilde{\xi}_{j},\tilde{\xi}_{k}\right]%&=\sum_{k=1}^n \left(\sum_{j=1}^{k-n_0} \Cov \left[\tilde{\xi}_{j},\tilde{\xi}_{k}\right]+\sum_{j=k-n_0}^{k-1} \Cov \left[\tilde{\xi}_{j},\tilde{\xi}_{k}\right]\right)
%\\
&\leq \sum_{k=1}^n \sum_{j=1}^{k-1} C e^{-\gamma e^{(\log (k-j))^{1/4}}}%+\sum_{j=k-n_0\wedge 0}^{k-1} \left(\big|\Cov_{\nu} \left[\tilde{\xi}_{1},\tilde{\xi}_{k-j+1}\right]\big|+Ce^{-\gamma e^{(\log j)^{1/4}}}\right)\right)
\\
%&\leq \sum_{k=1}^n \sum_{j=1}^k  \frac{C}{j^3}+n_0\sum_{k=n_0}^n \left( C(n_0)+C e^{-\gamma e^{(\log k-n_0)^{1/4}}}\right)
%\\
%&\leq \frac{C'}{n}+n_0 \left( C(n_0)\cdot n+\frac{C}{n}\right)
&\leq n\sum_{j=1}^\infty Ce^{-\gamma e^{(\log j)^{1/4}}}.
\end{align*} 
%where $C(n_0)=\max_{j\in \{1,n_0\}} \big|\Cov_{\nu} \left[\tilde{\xi}_{1},\tilde{\xi}_{k-j+1}\right]\big|$.

So, $\Var_{\sigma}\left[\frac{S_n}{n}\right]\leq \frac{C}{n}$. Applying Chebychev Inequality and Borel-Cantelli Lemma, we deduce that $\frac{S_{n^2}}{n^2}$ converges a.s. To prove the convergence of the complete sequence, we need to control $\Var_{\sigma}\left[\frac{S_p}{p}-\frac{S_n}{n}\right]$ with $p=p(n)=\lfloor\sqrt{n}\rfloor^2$. It is easy to check that $n-p\leq 2\sqrt{n}\leq p$ for $n$ large enough. We introduce the following quantity
\[D_{n,p}=\tilde{\xi}_{p+1}+\ldots+\tilde{\xi}_{n}.\]
We have that

\begin{align*}
\Var_{\sigma}\left[\frac{S_p}{p}-\frac{S_n}{n}\right]&=\Var_{\sigma} \left[\left(\frac{1}{p}-\frac{1}{n}\right) S_p-\frac{1}{n} D_{n,p}\right]
\\
&=\left(\frac{1}{p}-\frac{1}{n}\right)^2 \Var_{\sigma} \left[S_p\right]+\frac{1}{n^2} \Var_{\sigma} \left[D_{n,p}\right]-2\left(\frac{1}{p}-\frac{1}{n}\right)\frac{1}{n}\Cov_{\sigma} \left[S_p,D_{n,p}\right].
\end{align*}
The two first terms are dominated by $\frac{C(n-p)}{n^{2}}\leq\frac{2C}{n^{3/2}}$. Using the same computations as above, we have that
\begin{align*}
\Cov_{\sigma}[S_p,D_{n,p}]= \sum_{i=1}^p\sum_{j=p+1}^n \Cov_{\sigma} [\tilde\xi_i,\tilde\xi_j]\leq C p.
\end{align*}
This yields
\begin{align*}
\Var_{\sigma}\left[\frac{S_p}{p}-\frac{S_n}{n}\right]&\leq \frac{C}{n^{3/2}},
\end{align*}
and by the same arguments as previously, $\frac{S_{p(n)}}{p(n)}-\frac{S_n}{n}$ converges almost surely to $0$, as well as $\frac{S_{p(n)}}{p(n)}$. We deduce that almost surely, $\frac{S_n}{n}$ converges to $0$. Moreover, by Lemma~\ref{momentxinu}, $\frac{1}{n}\sum_{i=1}^n\E_\sigma[\xi_i]$ is the Cesaro mean of a convergent sequence. It is also clear (e.g.\ by finite speed of propagation) that $\frac{X(\sigma_t)-\sum_{i=1}^n\xi_i}{t}$ converges a.s.\ to $0$. To identify the limit in \eqref{eq:lgn} and conclude, we % and we deduce~\eqref{eq:lgn}. To conclude on the convergence of $\frac{X(\sigma_t)}{t}$ we 
observe that
\[\frac{d}{dt} \E_{\sigma}\left[X(\sigma_t)\right]=p\cdot\tilde{\mu}_t^{\sigma}(\tilde{\sigma}(1)=0)-q,\]
which converges to the announced velocity.

To prove~\eqref{eq:tlc}, we apply Theorem~\ref{thm:TCLgal} with $\Phi(k)=\exp(-\exp((\log k)^{1/4}))$, $c^*=\frac{d^*}{\overline{v}}$ and $\nu$ the measure from~Theorem~\ref{thm:coupling}. The required hypotheses were proved in Subsection~\ref{section:hypotheses}. \end{proof}

\begin{appendices}

\section{A central limit theorem}
 The following result is a general version of the central limit proved in~\cite{GLM15}, which was based on a Bolthausen result~\cite{bolt}, itself inspired by Stein's method. It gives a central limit theorem for random variables with some mixing conditions but without stationarity hypotheses. 
  
  \begin{theorem}\label{thm:TCLgal}
  Let $(\sigma_t)$ be a Markov process and $(\mathcal F_t)$ the adapted filtration. Let $(X_i)_{i\geq 1}$ be real random variables satisfying the following hypotheses:
 \begin{enumerate}
 \item\begin{enumerate}
 \item\label{hyp_moment} $\displaystyle \sup_{\sigma}\sup_{n\in\N}\E_{\sigma}[X_n^2]<\infty$;
 \item \label{hyp_adapted} for every $i\geq 1$, $X_i$ is measurable w.r.t.\ $\mathcal{F}_i$;
 \item \label{hyp_markov} for every $k,n\geq 1$, $f:\R^n\to\R$ measurable such that $\sup_\sigma\E_\sigma[f(X_1,\ldots,X_n)]<\infty$, for all initial $\sigma$, we have the Markov property
 \begin{equation}
\E_\sigma[f(X_k,\ldots,X_{k+n-1})\mid\mathcal{F}_{k-1}]=\E_{\sigma_{k-1}}[f(X_1,\ldots,X_{n})];
 \end{equation}
 \end{enumerate}
 \item There exist a decreasing function $\Phi$, constants $ {C}, c^*\geq 1$ and $v\in\R$ and a measure $\nu$ such that%linear operator $ {\E_\nu} {:\R^\Omega}\to \R$ such that
 \begin{enumerate}
 \item\label{hyp_Phi} $\displaystyle \lim_{n\to\infty} e^{(\log n)^2} \Phi(n)=0$;
 \item  {for every $i\geq 1$, $ {\E_\nu}[X_i]=v$;}
  \item  for every $k$, $f:\R\to\R$ s.t.\ $e^{-|x|}f(x)\in L^1(\R)$,
  \begin{equation}
  \sup_{\sigma}|\E_{\sigma}\left[f(X_{k})\right]-\E_\nu[f(X_1)]|\leq C\Phi(k);
  \label{hyp_esp}
  \end{equation}
   \item  for every $k,n$ such that $k {\geq}c^* n$,
 \begin{equation}
 \sup_{\sigma}|\E_{\sigma}\left[X_k X_{k+n}\right] - {\E_\nu}[X_1 X_{n+1}]| {\leq C}\Phi(k);
 \label{hyp_esp2}
 \end{equation}
\item for every $k,n$ such that $k>c^* n$ and any bounded function $F:\R^n\to \R$ 
 \begin{equation}
 \sup_{\sigma}\big|\E_{\sigma}\left[F(X_k, \ldots, X_{k+n-1})\right] - {\E_\nu}\left[F(X_1,\ldots,X_{n})\right]\big|\leq C\|F\|_{\infty}\Phi(k).
 \label{hyp_espF}
 \end{equation}
 \end{enumerate}
 \end{enumerate}
% $\sigma^2=\sum_{n=1}^\infty \Cov_{\sigma}(X_1, X_n)$ 
Then, there exists $s\geq 0$ such that 
\[
\frac{\sum_{i=1}^n X_i-vn}{\sqrt{n}}\xrightarrow{\mathscr{L}} \mathcal{N}(0,s^2).
\]
\end{theorem}

 {\begin{lemma}\label{lem:covariances}
The hypotheses of Theorem~\ref{thm:TCLgal} imply that there exists $C'>0$ such that for every $i\leq j$,
\begin{align}
\sup_\sigma\left|\Cov_\sigma[X_i,X_j]\right|&\leq C'\Phi(j-i),\label{hyp_cov2}\\
\left|\Cov_\nu[X_i,X_j]\right|&\leq C'\Phi(j-i),\label{hyp_covnu}\\
\sup_\sigma\left|\Cov_\sigma[X_i,X_j]-\Cov_\nu[X_1,X_{j-i+1}]\right|&\leq C'\Phi(i/c^*)\wedge\Phi(j-i).\label{eq:covariances}
\end{align}
\end{lemma}}

\begin{proof}
 {On the one hand, by Hypothesis \eqref{hyp_adapted} and the Markov property, we have
\begin{align*}
\Cov_\sigma[X_i,X_j]&=\E_\sigma\left[X_i\left(\E_{\sigma_i}[X_{j-i}]-\E_\sigma[X_j]\right)\right]\\
&=O(\Phi(j-i)),
\end{align*}
where the last equality comes from \eqref{hyp_esp} applied twice to $f(x)=x$ and Hypothesis~\eqref{hyp_moment}. The same strategy also shows $\Cov_\nu[X_1,X_{j-i+1}]=O(\Phi(j-i))$.}

 On the other hand, if $i\geq c^*(j-i+1)$, we can apply \eqref{hyp_esp2} and again \eqref{hyp_esp} to get
\[
\Cov_\sigma[X_i,X_j]=\Cov_\nu[X_1,X_{j-i+1}]+O(\Phi(i)).\qedhere
\]
\end{proof}

 \begin{proof}[Proof of Theorem~\ref{thm:TCLgal}]
 { We begin by showing that 
 \begin{equation}\label{eq:limitvar}
\lim_{n\to\infty}\frac{1}{n}\Var_\sigma\left[\sum_{i=1}^nX_i\right]=
\Var_\nu[X_1]+2\sum_{k=2}^\infty\Cov_\nu[X_1,X_k]=:s^2<\infty.
\end{equation}
We have $s^2<\infty$ by $\eqref{hyp_covnu}$. Moreover,
 by \eqref{hyp_esp} applied to $f(x)=x^2$, 
\[
\frac{1}{n}\sum_{i=1}^n\Var_\sigma[X_i]=\frac{1}{n}\sum_{i=1}^n\left(\Var_\nu[X_1]+O(\Phi(i))\right)=\Var_\nu[X_1]+O(1/n).
\]
Similarly, by \eqref{eq:covariances}, 
\begin{align*}
\frac{1}{n}\sum_{i<j\leq n}\Cov_\sigma[X_i,X_j]&=\frac{1}{n}\sum_{i<j\leq n}\left(\Cov_\nu[X_1,X_{j-i+1}]+O(\Phi(i/c_*)\Phi(j-i))\right)\\
&=\sum_{k=2}^\infty\Cov_\nu[X_1,X_k]+o(1),
\end{align*}
which concludes the proof of \eqref{eq:limitvar}.
}
 
  {The proof then depends on the value of $s^2$: if $s^2=0$, Chebychev's inequality shows that $\frac{\sum_{i=1}^n X_i-vn}{\sqrt{n}}\xrightarrow{}0$ in probability.}
 
  {If $s^2>0$, the proof appeals to the variation of Stein's method found in \cite{bolt}.} For every $i$, let $Y_i=X_i-v$. In particular for every $i$, $ {\E_\nu}[Y_i]=0$. First, we prove the result for random variables $Y_i$ bounded by $C_Y$ and satisfying the hypotheses of the theorem. Let $\ell_n=n^{1/3}$ and 
 \begin{align*}
  {S_n=\sum_{k=1}^nY_k},\quad S_{j,n}=\sum_{k=1}^n \1_{|k-j|\leq \ell_n} Y_k\quad\text{and}\quad
 \alpha_n= \sum_{j=1}^n \E_{\sigma}\left[Y_jS_{j,n}\right].
 \end{align*}
The theorem's hypotheses and the boundedness of the $Y_i$'s imply that there exists $C'$ such that
\begin{enumerate}
%\item For every $j<k$ we have that
%\begin{align} \sup_{\sigma}|\Cov_{\sigma}\left[Y_j,Y_k\right]| {\leq C'}\Phi(k-j).
%\end{align}
\item For every $i\leq j<k \leq l$ we have that
\begin{align}
\label{hyp_cov4_1} 
\sup_{\sigma}|\Cov_{\sigma}\left[Y_iY_j,Y_kY_l\right]| {\leq C'}\Phi(k-j)\text{ if }(k-j) \geq  {c^*(l-k)}.
\end{align}
% if $j=k<l$ then $\Phi(l-j)$; if i<j=k then
\item For every $i\leq j\leq k < l$ we have that
\begin{align}
\label{hyp_cov4_2} 
\sup_{\sigma}|\Cov_{\sigma}\left[Y_iY_k,Y_jY_l\right]|& {\leq C'}\Phi(l-k),
\\
\label{hyp_cov4_3} 
\sup_{\sigma}|\Cov_{\sigma}\left[Y_iY_l,Y_jY_k\right]|& {\leq C'}\Phi(l-k).
\end{align}
% de maniere generale c'est Phi(l-le premier different)
%\item For every $i<j$ we have that%\noteO{Pareil que pour la remarque après le lemme 7.3 : à mon avis, tout ce qu'on a (et je pense que ça suffit), c'est $\leq C' \Phi(j-i)$, en combinant \eqref{hyp_cov2} et l'hypothèse \ref{hyp_espF} ou \ref{hyp_esp2}. On pourrait d'ailleurs préciser que c'est grâce à \ref{hyp_espF} qu'on peut contrôler $ {\E_\nu}[Y_1\ldots Y_n]$ par $C_Y^n$.}
%\begin{align}\label{hyp_cov_Psi}
%| {\E_\nu}(Y_i,Y_j)| {\leq C'}\Phi(j-i).
%\end{align}
\end{enumerate}
We just give details for the  {first item}. % {Applying hypotheses \eqref{hyp_adapted} and \eqref{hyp_markov}} (later referred to as ``Markov property'') at time $j$ and Cauchy-Schwarz inequality, we have that 
 %\begin{align*}
%\Cov_{\sigma}\left[Y_j,Y_k\right]&=\Cov_{\sigma}\big[ Y_j, {\E_{\sigma_j}\left[Y_{k-j} \right]}\big]\leq \sqrt{\E_{\sigma}\left[Y_j^2\right]} \sqrt{\E_{\sigma}\left[\left(\E_{\sigma_j} {\left[Y_{k-j} \right]}\right)^2\right]},
% \end{align*}
%and we conclude using hypotheses~\eqref{hyp_moment} and~\eqref{hyp_esp}  {with $f(x)=x$}. 
We have that
 \begin{align*}
\Cov_{\sigma}\left[Y_i Y_j,Y_k Y_l\right]&=\Cov_{\sigma}\big[Y_i Y_j,\E_{\sigma_j}\left[Y_{k-j} Y_{l-j}\right]- {\E_\nu}\left[Y_{1} Y_{l-k+1}\right]\big]\\
&\leq {2C_Y^2\sup_{\sigma'}\left|\E_{\sigma'}\left[Y_{k-j} Y_{l-j}\right]- {\E_\nu}\left[Y_{1} Y_{l-k+1}\right]\right|}%\sqrt{\E_{\sigma}\left[Y_i^2 Y_j^2\right]} \sqrt{\E_{\sigma}\left[\left(\E_{\sigma_j}\left[Y_{k-j} Y_{l-j}\right]- {\E_\nu}\left[Y_{1} Y_{l-k+1}\right]\right)^2\right]},
 \end{align*}
 and we conclude using hypothesis~\eqref{hyp_esp2}. %and the fact that the $(Y_i)$ are bounded. %on pourrait aussi supposer moment 4
The others cases are obtained applying the Markov property at time $k$ and hypothesis~\eqref{hyp_esp}.% Finally the last estimate~\eqref{hyp_cov_Psi} comes from~\eqref{hyp_cov2} and~\eqref{hyp_esp2}.
 
Thanks to~\eqref{hyp_cov2}, we have that 
 \begin{align*}
 |\Var_{\sigma}[S_n]-\alpha_n|&=\Big|\sum_{|i-j|>\ell_n}\Cov_{\sigma}[Y_i,Y_j] {-\sum_{|i-j|\leq\ell_n}\E_\sigma[Y_i]\E_\sigma[Y_j]}\Big|\\
 & {=O(n^2  \Phi(\ell_n)+1)= O(1)};
 \end{align*}
  so $\alpha_n=\Var_{\sigma} [S_n] {(1+o(1))}$ and it is enough to prove that $\frac{S_n}{\sqrt{\alpha_n}}$ is asymptotically normal.
 %changer al phrase, 
 
 The main idea is to use a lemma from~\cite{bolt}, itself inspired from Stein method, which is the following one.
 \begin{lemma}[Lemma~2 in~\cite{bolt}] Let $(\nu_n)_{n\in\N}$ be a sequence of probability measures in $\R$ with 
 \begin{enumerate}
 \item[(a)] $\displaystyle \sup_n\int |x|^2\nu_n(dx)<\infty$
 \item[(b)] \label{itemb}$\displaystyle \lim_{n\to\infty}\int (i\lambda-x)e^{i\lambda x}\nu_n(dx)=0$ for all $\lambda\in\R$.
 \end{enumerate}
 Then the sequence $(\nu_n)$ converges to the standard normal law.
 \end{lemma}
 The proof can be found in~\cite{bolt}. We want to apply this lemma to the distributions of the random variables $\left(\frac{S_n}{\sqrt{\alpha_n}}\right)_{n\in\N}$.  {By \eqref{hyp_esp} applied to $f(x)=x$, the Markov property applied at time $k$ and the boundedness of the $Y_i$'s, $\E_\sigma[Y_kY_j]=O(\Phi(j-k))$. The verification of item (a) follows.}
 
  {We now need to check item (b).}
Let $\lambda\in\R$ and $A_n=\left(i\lambda-\frac{S_n}{\sqrt{\alpha_n}}\right)e^{i\lambda\frac{S_n}{\sqrt{\alpha_n}}}$. To prove that $\E_{\sigma}\left[A_n\right]$ goes to  $0$ we part $A_n$ in three terms:
 \begin{align*}
 A_n^{(1)}&=i\lambda e^{i\lambda\frac{S_n}{\sqrt{\alpha_n}}} \left(1-\frac{1}{\alpha_n}\sum_{j=1}^n Y_j S_{j,n}\right),
 \\
 A_n^{(2)}&= {-}\frac{1}{\sqrt{\alpha_n}}e^{i\lambda\frac{S_n}{\sqrt{\alpha_n}}}\sum_{j=1}^n Y_j\left(1-e^{ {-}i\lambda\frac{S_{j,n}}{\sqrt{\alpha_n}}}-i\lambda\frac{S_{j,n}}{\sqrt{\alpha_n}}\right),
 \\
  A_n^{(3)}&= {-}\frac{1}{\sqrt{\alpha_n}}\sum_{j=1}^n Y_je^{i\lambda\frac{S_n-S_{j,n}}{\sqrt{\alpha_n}}} .
 \end{align*}
The decaying of the two first terms come easily from the estimates on the covariances.
 \begin{align*}
 \E_{\sigma}\left[ {\left|A_n^{(1)}\right|}^2\right]&=\lambda^2\Var_{\sigma}\left[ \frac{1}{\alpha_n}\sum_{j=1}^n Y_j S_{j,n}\right]=\frac{\lambda^2}{\alpha_n^2}\sum_{i,j=1}^n \Cov_{\sigma}\left[Y_i S_{i,n},Y_j S_{j,n}\right] 
 \\
&=\frac{\lambda^2}{\alpha_n^2}\sum_{i,j=1}^n~~ \sum_{\substack{k:|k-i|\leq \ell_n\\l:|l-j|\leq \ell_n}}\Cov_{\sigma}\left[Y_i Y_k,Y_j Y_l\right].
 \end{align*}
 
If $|i-j|\geq (c^*+2)\ell_n$, then w.l.o.g $i\leq k<j\leq l$ and $(j-k)\geq c^*\ell_n \geq c^*(l-j)$; so applying~\eqref{hyp_cov4_1}, we have that
 \begin{align*}
\Cov_{\sigma}\left[Y_i Y_k,Y_j Y_l\right]= O(\Phi(j-k)).
 \end{align*}
If we are not in the previous case then w.l.o.g $i\leq k\leq j<l$) and applying~\eqref{hyp_cov4_2}
then
\begin{align*}
|\Cov_{\sigma}\left[Y_i Y_k,Y_j Y_l\right]|= O(\Phi(l-j)).
\end{align*}
Consequently,
\begin{align*}
 \E_{\sigma}\left[\left|A_n^{(1)}\right|^2\right]
\leq C\frac{\lambda^2}{\alpha_n^2}\left(n^3 \sum_{m=\ell_n}^n \Phi(m)+ n\ell_n^2\sum_{m=0}^{\ell_n}\Phi(m)\right)=o(1).
 \end{align*}
Using the Taylor expansion of the exponential  {and the fact that $\E_\sigma[S^2_{j,n}]=O(\ell_n)$ thanks to \eqref{hyp_cov2},} we easily control the term $A_n^{(2)}$:
\[
 \E_{\sigma}\left[| A_n^{(2)}|\right]\leq\frac{ {C\lambda^2}}{\sqrt{\alpha_n}}\sum_{j=1}^n \E_{\sigma}\left[|Y_j|\cdot \frac{S_{j,n}^2}{\alpha_n}\right]\leq \frac{ {C\lambda^2}}{\sqrt{\alpha_n}}\cdot \frac{ {C_Y}}{\alpha_n}\cdot n\cdot {\ell_n}=O\left( {\frac{\ell_n}{\sqrt{n}}}\right)=o(1).
 \]

 To study $\E[A_n^{(3)}]$, Bolthausen uses the stationarity of the field but we do not have this property here. We need to analyse carefully this term as it is done in~\cite{GLM15}. First, using the boundedness of the variables $(Y_j)$ and the fact than $\ell_n\ll\sqrt{\alpha_n}$ we just need to prove that 
 \begin{align}\label{limA3}
 \lim_{n\to\infty}\frac{1}{\sqrt{\alpha_n}}\sum_{j=\ell_n}^{n} \E_{\sigma}\left[Y_je^{i\lambda\frac{S_n-S_{j,n}}{\sqrt{\alpha_n}}}\right]=0,~\forall \lambda \in \R.
 \end{align}
 The usefulness of considering the sum from $\ell_n$ instead of $0$ will appear later. We fix $M$ a number (which will eventually depend on $n$) and we develop the exponential in two parts (its partial sum and its remainder):
\begin{align*}
e^{i\lambda\frac{S_n-S_{j,n}}{\sqrt{\alpha_n}}} &=T_{j,n}^M(\lambda)+R_{j,n}^M(\lambda)~\text{ with } \left\{ 
\begin{array}{ll}
T_{j,n}^M(\lambda)&= \displaystyle \sum_{m=0}^M \frac{(i\lambda)^m}{m!} \left(\frac{S_n-S_{j,n}}{\sqrt{\alpha_n}}\right)^m
\\
R_{j,n}^M(\lambda)&= \displaystyle \sum_{m=M+1}^\infty \frac{(i\lambda)^m}{m!} \left(\frac{S_n-S_{j,n}}{\sqrt{\alpha_n}}\right)^m.
\end{array}
\right. 
\end{align*}
Let us first analyse the contribution of $T_{j,n}^M(\lambda)$ to~\eqref{limA3}:
\begin{align*}
\frac{1}{\sqrt{\alpha_n}}\sum_{j=\ell_n}^{n} \E_{\sigma}\left[Y_jT_{j,n}^M(\lambda)\right] &= \frac{1}{\sqrt{\alpha_n}}\sum_{j=\ell_n}^{n}\sum_{m=0}^M\frac{1}{m!}\left(\frac{i\lambda}{\sqrt{\alpha_n}}\right)^m \E_{\sigma}\Bigg[Y_j \Big(\sum_{\substack{1\leq i\leq n\\ |i-j| {>}\ell_n}} Y_i\Big)^m\Bigg] 
\\
&= \frac{1}{\sqrt{\alpha_n}}\sum_{j=\ell_n}^{n}\sum_{m=0}^M\frac{1}{m!}\left(\frac{i\lambda}{\sqrt{\alpha_n}}\right)^m \sum_{i_1,\ldots,i_m\in\tau_j^{(m)}} \E_{\sigma}\left[Y_j \prod_{k=1}^m  Y_{i_k}\right] 
\end{align*}
where $\tau_j^{(m)}=\big\{(i_1,\ldots,i_m)\in\{1,\ldots,n\}^m\text{ such that }\forall k,~|i_k-j|\geq \ell_n\big\}$.

\begin{lemma}\label{lem:combi} If $m\leq \frac{\log n}{7\log(c^*+1)}$, then for any $j\in\{\ell_n,\ldots,n\}$ and any $(i_1,\ldots,i_m)\in{\tau_j^{(m)}}$, 
\[
\E_{\sigma}\left[Y_j \prod_{k=1}^m  Y_{i_k}\right]=O(n^{C'_Y}\Phi(\ell_n)),
\] 
where $C'_Y=\frac{\log C_Y}{7\log c^*}$.
\end{lemma}

\begin{proof} Let $m\leq  \frac{\log n}{7\log  {(c^*+1)}}$, $j\in\{\ell_n,\ldots,n\}$ and $(i_1,\ldots,i_m)$ such that $\forall k,~|i_k-j|\geq \ell_n$. We have several cases

\begin{enumerate}
\item if $i_m\leq j-\ell_n$, applying Markov property at time $i_m$ and hypothesis~\eqref{hyp_esp}, we have that
\begin{align*}
\left|\E_{\sigma}\left[Y_j \prod_{k=1}^m  Y_{i_k}\right]\right|\leq\E_{\sigma}\left[\prod_{k=1}^{m-1}|Y_{i_k}|\cdot
\left| \E_{\sigma_{i_m}}\left[Y_{j-i_m}\right] \right|\right] =  \cdot O(C_Y^m\Phi(j-i_m))
= O(n^{C'_Y}\Phi(\ell_n)).
\end{align*}
\item if there exists $b\leq m-1$ such that $i_b\leq j-\ell_n<j<j+\ell_n\leq i_{b+1}$ then we distinguish two cases:
\begin{enumerate}
\item\label{condition_lemmecombi} If, for all $k\geq b+2$, $i_k-i_{k-1}\leq (c')^{m-k} \sqrt{\ell_n}$ (with $c'=c^*+1$) then $ c^*(i_m-i_{b+1})\leq c^*(1+c'+\cdots+(c')^{m-3})\sqrt{\ell_n}\leq (c')^{m-2}\sqrt{\ell_n} {\ll} \ell_n \leq i_{b+1}-j$. So, using Markov property at time ${j}$ and hypothesis~\eqref{hyp_espF}, we have that
\begin{align*}
\E_{\sigma}\left[Y_j\prod_{k=1}^m Y_{i_k}\right]
&=\E_{\sigma}\left[Y_j\prod_{k=1}^b Y_{i_k}
\E_{\sigma_{t_j}}\left[
Y_{i_{b+1}-j}\ldots Y_{i_{m}-j}\right] \right] 
\\
&=\E_{\sigma}\left[ Y_j\prod_{k=1}^b Y_{i_k}\right]\left( {\E_\nu}\left[
 Y_{1}\ldots Y_{i_{m}-i_{b+1}}\right]+O\left(C_Y^{m-(b+1)}\Phi(i_{b+1}-j)\right)\right)
\\
% = C_Y^{b}\Phi(\ell_n) (C_Y^{m-(b+1)}+ O\left(C_Y^{m-(b+1)}\Phi(i_{b+1}-j)\right))
&=O(n^{C'_Y}\Phi(\ell_n)),
\end{align*}
where the last equality comes from the first case where all the indices are smaller that $j-\ell_n$.% {and the property $| {\E_\nu}[Y_{i_1}\ldots Y_{i_n}]|\leq C_Y^n$ which can be deduced from hypothesis \eqref{hyp_espF}}.
\item Else, we denote by $k^*=\max \{k\geq b+1 :i_{k+1}\geq i_k +(c')^{m-k-1}\sqrt{\ell_n}\}$. 
Then $c^*(i_m-i_{k^*+1})\leq c^*(1+c'+\cdots+(c')^{m-k^*-2})\leq (c')^{m-k^*-1}\sqrt{\ell_n} \leq i_{k^*+1}-i_{k^*}$. 
Applying Markov property at time $i_{k^*}$  
\begin{align*}
\E_{\sigma}\left[Y_j\prod_{k=1}^m Y_{i_k}\right] &=\E_{\sigma}\left[ Y_j\prod_{k=1}^{k^*} Y_{i_k}\right]\left( {\E_\nu}\left[
 Y_{1}\ldots Y_{i_{m}-i_{k^*+1}}\right]+O\left(C_Y^{m}\Phi(i_{k^*+1}-i_{k^*})\right)\right).
\end{align*}
If all the indices between $b+1$ and $k^*$ satisfy~\eqref{condition_lemmecombi} then we are reduced to the previous case. If not, we iterate until we are reduced to the previous case or to the first case (all indices smaller than $j-\ell_n$).
\end{enumerate}
\item If $i_1\geq j+\ell_n$, then we do the same computation as in the previous case with $b=0$ and we conclude using that $\E_{\sigma}[Y_j]=O(\Phi(j))=O(\Phi(\ell_n))$ because $j\geq \ell_n$.\qedhere
\end{enumerate}
\end{proof}

We choose $M=\frac{\log n}{7\log(c^*+1)}$, we use Lemma~\ref{lem:combi} and thanks to the hypotheses~\eqref{hyp_Phi} on $\Phi$ we get that, for every $\lambda\in\R$,
\begin{align*}
\left|\frac{1}{\sqrt{\alpha_n}}\sum_{j=\ell_n}^{n} \E_{\sigma}\left[Y_jT_{j,n}^M(\lambda)\right]\right| &\leq \frac{1}{\sqrt{\alpha_n}}\cdot n\cdot n^{\log n}\cdot O(n^{C'_Y}\Phi(\ell_n))\cdot e^{\frac{|\lambda|}{\sqrt{\alpha_n}}}=o(1).
\end{align*}
%\[ \lim_{n\to\infty}\frac{1}{\sqrt{\alpha_n}}\sum_{j=\ell_n}^{n} \E\left[Y_jT_{j,n}^M(\lambda)\right]=0.\]

We now have to understand the contribution of $R_{j,n}^M(\lambda)$ to~\eqref{limA3}. We fix $L$ a number (which will eventually depend on $n$).

\begin{align*}
R_{j,n}^M(\lambda) {=\tilde{R}_{j,n}^M(\lambda)+\bar{R}_{j,n}^M(\lambda)}&= \sum_{m=M+1}^\infty \frac{(i\lambda)^m}{m!} \left(\frac{S_n-S_{j,n}}{\sqrt{\alpha_n}}\right)^m\1_{\left\{ \big|\frac{S_n-S_{j,n}}{\sqrt{\alpha_n}}\big|\leq L\right\} }
\\
&~~+\sum_{m=M+1}^\infty \frac{(i\lambda)^m}{m!} \left(\frac{S_n-S_{j,n}}{\sqrt{\alpha_n}}\right)^m\1_{\left\{ \big|\frac{S_n-S_{j,n}}{\sqrt{\alpha_n}}\big|>L\right\} }.
\end{align*}
Using classical formulas on the remainder of the exponential series (and Cauchy-Schwarz inequality for the second inequality), we obtain that
\begin{align}
\frac{1}{\sqrt{\alpha_n}}\sum_{j=\ell_n}^{n} \E_{\sigma}\left[Y_j\tilde{R}_{j,n}^M(\lambda)\right]&\leq  \frac{n}{\sqrt{\alpha_n}}\cdot C_Y\cdot \frac{|\lambda|^{M+1}L^{M+1}}{M!}e^{|\lambda|L}, 
\label{eqR1}
\\
\frac{1}{\sqrt{\alpha_n}}\sum_{j=\ell_n}^{n} \E_{\sigma}\left[Y_j\overline{R}_{j,n}^M(\lambda)\right]&\leq \frac{n}{\sqrt{\alpha_n}}\cdot C_Y\cdot \max_j \E_{\sigma}\left[e^{2|\lambda| \frac{|S_n-S_{j,n}|}{\sqrt{\alpha_n}}}\right]^{1/2}\P_{\sigma}\left( \big|\frac{S_n-S_{j,n}}{\sqrt{\alpha_n}}\big| >L\right)^{1/2}.
\label{eqR2}
\end{align}
The right-hand side of \eqref{eqR1} will be taken care of by an appropriate choice of $L$. The right-hand side of \eqref{eqR2} calls for a better understanding of $S_n$ and $S_{j,n}$. 
\begin{lemma}\label{lem:moment_expo_sn} There exists $c>0$ such that for any $n$ large enough and any $\beta=O(\sqrt{n}\ell_n^{-1})$,
\[
 \E_{\sigma}\left[e^{\beta \frac{|S_n|}{\sqrt{\alpha_n}}}\right]\leq 2e^{c\beta^2}.
 \] 
\end{lemma}

\begin{proof}
 {For $A\subset\{1,\ldots,n\}$, let $Z_A=\sum_{k\in A}Y_k$. It is immediate that $|Z_A|\leq C_Y|A|$.}

We partition the interval $\{  {1},\ldots,n\}$ into blocks $B$ of $\ell_n$ successive integers. Let $B$ be a block such that the smaller index $s_B$ of $B$ is larger than $ c^* \ell_n$. % and $Z_B$ the associate sum. The random variables $(Y_i)_i$ being bounded by $C_Y$, we have that
%\[
%\frac{|Z_B|}{\sqrt{n}}\leq C_Y\ell_n n^{-1/2}.
%\] 
Let $t\in \{0,\cdots, s_B- c^* \ell_n\}$. We can apply hypothesis~\eqref{hyp_espF} with the function $F:(x_1,\ldots,x_{\ell_n})\mapsto \exp(\frac{\beta\sum x_i {\wedge (\ell_nC_Y)}}{\sqrt{n}})$ (and the parameter $k=s_{B}-t>c^*\ell_n$):
\begin{align*}
\E_{\sigma}\left[\exp\left(\frac{\beta Z_{B}}{\sqrt{n}}\right)\big| \mathcal{F}_{t}\right]&=\E_{\sigma_{t}}\left[\exp\left(\frac{\beta Z_{B-t}}{\sqrt{n}}\right)\right]
\\
&= {\E_\nu}\left[\exp\left(\frac{\beta Z_{B}}{\sqrt{n}}\right)\right]+O(\|F\|_\infty\Phi(s_{B}-t))
\\
&= {\E_\nu}\left[\exp\left(\frac{\beta Z_{B}}{\sqrt{n}}\right)\right]+O(\exp(\beta C_Y \ell_n n^{-1/2})\Phi(\ell_n))
\end{align*}
where $Z_{B-t}$ designates the sum of $\ell_n$ variables starting from $s_B-t$. Furthermore, we have that
\begin{align*}
 {\E_\nu}\left[\exp\left(\frac{\beta Z_{B}}{\sqrt{n}}\right)\right]&= {\E_\nu}\left[1+\frac{\beta Z_{B}}{\sqrt{n}}+\frac{\beta^2 Z_{B}^2}{2n}+O\left(\frac{\beta^3 Z_{B}^3}{n^{3/2}}\right)\right]
\\
&=1+\frac{\beta^2}{2n} {\E_\nu}\left[Z_{B}^2\right]+O\left(\beta^3 \ell_n n^{- {3/2}}  {\E_\nu}\left[Z_{B}^2\right]\right).
\end{align*}
 {Equation~\eqref{hyp_cov2}} tells us that $ {\E_\nu}[Z_{B}^2]=O(\ell_n)$ so%, if $\beta= \frac{\sqrt{n}}{\ell_n}$ then 
\begin{align}
\label{borne1bloc}
\E_{\sigma}\left[\exp\left(\frac{\beta Z_{B}}{\sqrt{n}}\right)\big| \mathcal{F}_{t}\right]\leq 1+c\frac{\beta^2\ell_n}{n}.
\end{align}

Now we consider $\mathcal{B}$ (and $S_{\mathcal{B}}$ the associate sum) a set of blocks such that the distance between two blocks $B$ and $B'$ of $\mathcal B$ is larger than $ c^* \ell_n$; in other words, if $t_B$ is the maximal index in $B$ and $s_{B'}$ is the minimal one in $B'$ then $s_{B'}-t_B >  c^* \ell_n$. Let $B$ and $B'$ be two such blocks of $\mathcal B$. Markov property and Equation~\eqref{borne1bloc} (with $t=t_B$) implies that
\begin{align*}
\E_{\sigma}\left[\exp\left(\beta \frac{Z_{B}+Z_{B'}}{\sqrt{n}}\right)\right]&= \E_{\sigma}\left[\exp\left(\frac{\beta Z_{B}}{\sqrt{n}}\right)\E_{\sigma_{t_B}}\left[\exp\left(\frac{\beta Z_{B'}}{\sqrt{n}}\right)\big| \mathcal{F}_{t_B}\right]\right]\\
& {\leq}\E_{\sigma}\left[\exp\left(\frac{\beta Z_{B}}{\sqrt{n}}\right)\right]\left(1+c\frac{\beta^2\ell_n}{n}\right).
\end{align*} 
By definition of $\mathcal B$ we clearly have that $|\mathcal B|\leq \frac{n}{ \ell_n}$ so after iteration, we conclude that
\[
\E_{\sigma}\left[\exp\left(\beta \frac{S_{\mathcal B}}{\sqrt{n}}\right)\right]\leq \left(1+c\frac{\beta^2\ell_n}{n}\right)^{|\mathcal B|}\leq \exp\left(\frac{c\beta^2 \ell_n  n}{ n \ell_n}\right)\leq \exp(c'\beta^2).
\]
Let us go back to the whole sum $S_n$; we can write $S_n=Z_0+S_{\mathcal B_1}+\ldots+S_{\mathcal B_{\lceil  c^* \rceil}}$ where the $\mathcal B_i$ are disjoint sets of blocks with the same property as $\mathcal B$ and $Z_0$ is the sum of the $ c^*  \ell_n$ first terms.  {By H\"{o}lder's inequality}, we have that
\[
\E_{\sigma}\left[\exp\left(\beta \frac{S_{n}}{\sqrt{n}}\right)\right]\leq \E_{\sigma}\left[\exp\left(\frac{\beta Z_0}{\sqrt{n}}\right)\right]
\prod_{i=1}^{\lceil  c^* \rceil} \left(\E_{\sigma}\left[\exp\left(\beta \frac{S_{\mathcal B_i}}{\sqrt{n}}\right)^{\lceil  c^* \rceil}\right]\right)^{1/\lceil  c^* \rceil}\leq \exp(c''\beta^2)
\]
where we use that $\E_{\sigma}\left[\exp\left(\frac{\beta Z_0}{\sqrt{n}}\right)\right]\leq 1+\frac{\beta C_Y \ell_n}{\sqrt{n}}$. 
We could proceed similarly to control $\E_{\sigma}\left[\exp\left(-\beta \frac{S_{n}}{\sqrt{n}}\right)\right]$ and finally $\E_{\sigma}\left[\exp\left(\beta \frac{|S_{n}-S_{j,n}|}{\sqrt{n}}\right)\right]$.%; we conclude thanks to the fact that $\alpha_n\approx n$.
 \end{proof} 
 
So, if $L=O(\sqrt{n}\ell_n^{-1})$ then we can apply Lemma~\ref{lem:moment_expo_sn} with $\beta=\epsilon L$ (with $\epsilon$ small enough) and use the exponential Chebychev inequality to obtain that 
\begin{align}\label{eq:chebychev_expo}
\P_{\sigma}\left( \big|\frac{S_n-S_{j,n}}{\sqrt{\alpha_n}}\big| >L\right)\leq e^{-\epsilon L^2} \E_{\sigma}\left[\exp\left(\beta\big|\frac{S_n-S_{j,n}}{\sqrt{\alpha_n}}\big|\right) \right] \leq e^{-\epsilon'L^2}.
\end{align}

We recall that we have previously chosen $M=\frac{\log n}{7\log(c^*+1)}$. Now we choose \[L=\frac{\log n}{7\log(c^*+1)\cdot(10\wedge(|\lambda|+2))}\] and we use Lemma~\ref{lem:moment_expo_sn} and Equation~\eqref{eq:chebychev_expo} in Equation~\eqref{eqR2} to conclude that 
%DETAILS ?
\[\lim_{n\to\infty}\frac{1}{\sqrt{\alpha_n}}\sum_{j=\ell_n}^{n} \E_{\sigma}\left[Y_jR_{j,n}^M(\lambda)\right]=0.\]
Putting everything together we conclude the proof for bounded random variables $Y_i$. We use a classical truncation argument to prove the result for non bounded random variables: let $T^N(x)$ be the truncation operator and $R^N(x)$ the remainder:
\begin{align*}
T^N(x)&=\max\{\min\{x,N\},-N\}
\\
R^N(x)&=x-T^N(x).
\end{align*}
Thanks to~\eqref{hyp_cov2}  we have that
\[
\E_{\sigma}\left[\left(\frac{\displaystyle \sum_{i=1}^n\left( R^{N}(Y_i)-\E_{\sigma}\left[R^{N}(Y_i)\right]\right)}{\sqrt{n}}\right)^2\right]=\frac{1}{n}\sum_{i,j}^n \Cov_{\sigma}\left[R^{N}(Y_i),R^{N}(Y_j)\right]
\] converges to $0$ as $N$ goes to $\infty$ uniformly in $n$. Using moreover that $\Var_{\sigma}\left[R^{N}(Y_i)\right]\to 0$ as $N$ goes to infty, uniformly in $i$, we can conclude that the central limit theorem is still valid for the unbounded variables.
\end{proof}

\section{Contact process estimates}\label{annexe_cp}
In the classical contact process introduced by Harris in~\cite{harris74} (with the convention chosen in section~\ref{section_pc}), zeros flip to ones at rate 1 without any constraint, ones flip to zeros at rate $\lambda>0$ if there is exactly one nearest neighbor zero and at rate $2\lambda$ if  the two nearest neighbors are zeros. Let $(\eta_t^A)$ be such a contact process starting from a non empty finite set $A$ of infected particles, that is, the particles in the set $A$ have value $0$ and the other ones have value $1$. We define 
$\tau^A=\inf\{t>0:\eta_t^A\equiv 1\}$ its extinction time and we denote by $X(\eta_t^A)$ the position of leftmost $0$, \textit{i.e.} the leftmost infected point, at time $t$.

In~\cite{DG83} we can find the following estimates:
\begin{theorem}[Durrett-Griffeath, 1983]\label{DGestimates}
For the contact process with infection rate $\lambda$, there exists $\lambda_c\in(0,\infty)$ such that for $\lambda>\lambda_c$ the followings hold. Firstly, there exist $C_1$ and $C_2$ such that for $A$ non empty finite set and $t\geq 0$
\begin{align}
\label{survival}\P(\tau^A=\infty)&>0\\
\label{survival_initial}\P(\tau^A<\infty)&\le C_1 e^{-C_2 |A|}\\
\label{small_cluster}\P(t<\tau^{\{0\}}<\infty)&\leq C_1 e^{-C_2 t}.
\end{align}
Secondly, there exists $v_{cp}$ such that for every $a<v_{cp}$, $t\ge 0$ and $|x|<v_{cp} t$
\begin{align}
\label{atleastlinear}\P(X(\eta_t)>-at|\tau^{\{0\}}=\infty)&\leq C_1 e^{-C_2 t}\\
\label{coupled_region}   \P(\eta_t^{\{0\}}(x)\neq \eta_t^{\Z}(x)|\tau^{\{0\}}=\infty)&\le C_1 e^{-C_2 t}.
\end{align}
\end{theorem}
From this theorem, we derive the following estimate (used in the proof of Corollary~\ref{cor:zerozone}). 

\begin{corollary}\label{zerosCP} For the threshold contact process defined in paragraph~\ref{section_pc}, if $q>\bar{q}:=\frac{2\lambda_c}{1+2\lambda_c}$ the following holds. For every $t>0$ and $l\in\N$, if $I_l$ is a subset of $[-v_{cp} t,v_{cp} t]$ with size $l$ (where $v_{cp}$ is the constant given by~Theorem~\ref{DGestimates}) then there exist $C_3,C_4$ such that
\begin{align}\label{eq_zerosCP}
\P\left(\forall x\in I_l,\eta_t^{\{0\}}(x)=1 |\tau^{\{0\}}=\infty\right)\leq  C_3 e^{-C_4 t\wedge l}.
\end{align}
\end{corollary}
\begin{proof} By obvious monotonicity between threshold contact process and contact process (taking $\lambda=\frac{q}{p}$) it is enough to prove the result for the contact process with parameter $\lambda/2$, which satisfies $\lambda/2>\lambda_c$. Let $t>0$, $l\in\N$ and $I_l$ a set of size $l$.
Using duality and~\eqref{survival_initial} we can obtain:
\begin{align}\label{eq2}
\P\left(\forall x\in I_l, \eta_t^{\Z}(x)=1\right)\leq C_1 e^{-C_2 l}.
\end{align}
If $I_l$ is included in $[-v_{cp} t,v_{cp} t]$ then, using~\eqref{coupled_region} and~\eqref{eq2}, we have that
\begin{align*}
&\P\left(\forall x\in I_l:\eta_t^{\{0\}}(x)=1 |\tau^{\{0\}}=\infty\right)\\ 
&~~~~~\leq \P\left(\forall x\in I_l:\eta_t^{\Z}(x)=1 |\tau^{\{0\}}=\infty\right)+ \P\left(\exists x\in I_l:\eta_t^{\{0\}}(x)\neq \eta_t^{\Z}(x) \big|\tau^{\{0\}}=\infty\right)\\
&~~~~~\leq C_1 e^{-C_2 l} +2v_{cp} t C_1 e^{-C_2 t} \\ %because l\leq \alpha t
&~~~~~\leq  C_3 e^{-C_4 t\wedge l}.\qedhere
\end{align*}
\end{proof}
\end{appendices}

\subsection*{Acknowledgements}
This work has been supported by the ERC Starting Grant 680275 MALIG, by the ANR projects LSD (ANR-15-CE40-0020), MALIN (ANR-16-CE93-0003) and PPPP (ANR-16-CE40-0016), and by the LABEX MILYON (ANR-10-LABX-0070) of  Université de Lyon, within the program "Investissements d'Avenir" (ANR-11-IDEX-0007) operated by the French National Research Agency (ANR). The authors also want to thank Fabio Martinelli for useful discussions.

\bibliographystyle{alpha}
%\bibliography{biblio_fa1f}

\newcommand{\etalchar}[1]{$^{#1}$}
\begin{thebibliography}{BCM{\etalchar{+}}13}

\bibitem[AD02]{AD}
D.~Aldous and P.~Diaconis.
\newblock The asymmetric one-dimensional constrained ising model: rigorous
  results.
\newblock {\em J. Stat. Phys.}, 107(5-6):945?975, 2002.

\bibitem[BCM{\etalchar{+}}13]{BCMRT13}
O.~Blondel, N.~Cancrini, F.~Martinelli, C.~Roberto, and C.~Toninelli.
\newblock Fredrickson-{A}ndersen one spin facilitated model out of equilibrium.
\newblock {\em Markov Process. Related Fields}, 19(3):383--406, 2013.

\bibitem[BD88]{BD88}
M.A. Burschka and R.~Dickman.
\newblock Nonequilibrium critical poisoning in a single-species model.
\newblock {\em Physics Letters A}, 127(3):132--137, 1988.

\bibitem[BFM78]{brower}
R.C. Brower, M.A. Furman, and M.~Moshe.
\newblock Critical exponents for the reggeon quantum spin model.
\newblock {\em Physics Letters}, 76B:213--219, 1978.

\bibitem[Blo13]{front_east}
Oriane Blondel.
\newblock Front progression in the {E}ast model.
\newblock {\em Stochastic Process. Appl.}, 123(9):3430--3465, 2013.

\bibitem[Bol82]{bolt}
E.~Bolthausen.
\newblock On the central limit theorem for stationary mixing random fields.
\newblock {\em Ann. Probab.}, 10(4):1047--1050, 1982.

\bibitem[DG83]{DG83}
Richard Durrett and David Griffeath.
\newblock Supercritical contact processes on {${\bf Z}$}.
\newblock {\em Ann. Probab.}, 11(1):1--15, 1983.

\bibitem[Dur80]{Dur80}
Richard Durrett.
\newblock On the growth of one-dimensional contact processes.
\newblock {\em Ann. Probab.}, 8(5):890--907, 1980.

\bibitem[GK11]{GK}
Olivier Garet and Aline Kurtzmann.
\newblock {\em De l'int{\'e}gration aux probabilit{\'e}s}.
\newblock R{\'e}f{\'e}rences sciences. Ellipses Marketing, 2011.

\bibitem[GLM15]{GLM15}
S.~Ganguly, E.~Lubetzky, and F.~Martinelli.
\newblock Cutoff for the east process.
\newblock {\em Comm. Math. Phys.}, 335(3):1287--1322, 2015.

\bibitem[Har74]{harris74}
T.~E. Harris.
\newblock Contact interactions on a lattice.
\newblock {\em Ann. Probability}, 2:969--988, 1974.

\bibitem[LPW09]{levin}
David~A. Levin, Yuval Peres, and Elizabeth~L. Wilmer.
\newblock {\em Markov chains and mixing times}.
\newblock American Mathematical Society, Providence, RI, 2009.
\newblock With a chapter by James G. Propp and David B. Wilson.

\bibitem[MV18]{MV}
T.~{Mountford} and G.~{Valle}.
\newblock Exponential convergence for the fredrikson-andersen one spin
  facilitated model.
\newblock {\em ArXiv e-prints 1609.01364}, 2018.

\bibitem[RS03]{RS}
F.~Ritort and P.~Sollich.
\newblock Glassy dynamics of kinetically constrained models.
\newblock {\em Advances in Physics}, 52(4):219?342, 2003.

\end{thebibliography}

\end{document}